\documentclass{amsart}

\usepackage{amssymb, amsthm}

\usepackage{txfonts}
\usepackage{pdfsync}
\usepackage{enumerate}

\newtheorem{theorem}{Theorem}[section]
\newtheorem{thm}[theorem]{Theorem}%[section]
\newtheorem{lem}[theorem]{Lemma}
\newtheorem{prop}[theorem]{Proposition}
\newtheorem{cor}[theorem]{Corollary}
\newtheorem{notat}[theorem]{Notation}

\newtheorem{fact}[theorem]{Fact}

\theoremstyle{definition}
\newtheorem{definition}[theorem]{Definition}
\newtheorem{exam}[theorem]{Example}

\theoremstyle{remark}
\newtheorem{rem}[theorem]{Remark}

\newtheorem{qu}[theorem]{Question}

\newcommand{\Add}{\mathop{{\rm Add}}\nolimits}

\newcommand{\RMod}{R\text{\rm -Mod}}
\newcommand{\Rmod}{R{\text{\rm -mod}}}
\newcommand{\RInj}{R\text{\rm -Inj}}

\newcommand{\ModR}{\text{\rm Mod-}R}

\newcommand{\modR}{\text{\rm mod-}R}
\newcommand{\Rproj}{R{\text{\rm -proj}}}
\newcommand{\Rfree}{R{\text{\rm -free}}}

\newcommand{\PGen}{\text{\rm PGen}}

\newcommand{\End}{\mathop{\rm End}}
\newcommand{\Hom}{\mathop{\rm Hom}}

\newcommand{\D}{\rm{D}}
\newcommand{\pp}{{\rm{pp}}}

\newcommand{\ppf}{{\rm{ppf}}}

\newcommand{\tor}{{\rm{t}}}

\newcommand{\cal}{\mathcal}

\renewcommand{\bar}{\overline}
\newcommand{\br}{\bar}

\renewcommand{\phi}{\varphi}
\renewcommand{\rho}{\varrho}

\newcommand{\cB}{\cal B}
\newcommand{\cC}{\cal C}
\newcommand{\cD}{\cal D}

\newcommand{\tC}{\tilde{C}}

\newcommand{\cF}{\cal F}
\newcommand{\cG}{\cal G}
\newcommand{\cK}{\cal K}
\newcommand{\cL}{\cal L}
\newcommand{\cP}{\cal P}

\newcommand{\cS}{\cal S}

\renewcommand{\to}{\longrightarrow}

\begin{document}
\dedicatory{To SK Jain}
\title{Strict Mittag-Leffler modules and purely generated classes}


%    author one information
\author{Philipp Rothmaler}
\address{CUNY, The Graduate Center, 365 Fifth Avenue, Room 4208, NY, NY 10016}
\curraddr{}
\email{philipp.rothmaler@bcc.cuny.edu}
\thanks{Supported in part by PSC CUNY Awards 64595-00 42 and 68835-00 46, 
BCCF Faculty Scholarship Support Grant (Spring Cycle 2018)}

%%    author two information
%\author{}
%\address{}
%\curraddr{}
%\email{}
%\thanks{}

%\subjclass[2000]{Primary }
%    The 2010 edition of the Mathematics Subject Classification is
%    now available.  If you are citing a classification from the
%    new scheme, use the following input coding instead.

\date{published in Contemporary Mathematics, Volume {\bf 751}, 2020. https://doi.org/10.1090/conm/751/15085}

\begin{abstract}
%As a sequel to \cite{MLII} we study  strict versions of $\cK$-Mittag-Leffler modules where now the role  definable subcategories, and present  newer ones, culminating ???  $\cK$-Mittag-Leffler modules.
We study versions of strict Mittag-Leffler modules relativized to a class $\cK$ (of modules), that is, \emph{strict} versions (in the technical sense of Raynaud and Gruson) of $\cK$-Mittag-Leffler modules, as investigated in the preceding paper, {\em Mittag-Leffler modules and definable subcategories},  in this very series. 
\end{abstract}

\maketitle
%Counterexamples for inclusions of sep'ty????
%
%Adjust numbering mode according to email 
%
%organize into subsections according to "going up" and "going down" preservation????
%
%???crazy indentation when you use [\rm (1)] (or [\rm (i)] in enumerations. In the latter it doesn't matter, but the former looks stupid in Thms. But if you don't use it, it gets italicized in Thms , which is annoying. ???
%
%???Should I look waht happens if separation is weakened to covering classes as in \cite{Tf} and \cite{MLII}???
%
%\bigskip

%The reason for $R_R\in  K$ implying f.g. $\to$ f.p. is Goodearls criterion tHAT EVERY FG SUBMODULE FACTOR THROUGH A FP. Where???

%\subsection{Introduction} 
Goodearl \cite{Goo} noticed (in different terms) that a module is $R_R$-Mittag-Leffler if and only if the inclusion of any of its finitely generated submodules factors through a finitely presented module. (This can be found, among other more or less relative Mittag-Leffler properties, in the preceding paper \cite{MLII}.) Raynaud and Gruson \cite{RG} characterized \emph{strict} Mittag-Leffler modules as those which can be mapped down to any finitely generated submodule in such a way that this map pointwise fixes the given finitely generated submodule and factors through a finitely presented module. There is a major difference here: in the strict version one must be able to map the \emph{entire} module each time one examines a finitely generated submodule. What is in common though is a certain local behavior concerning finitely generated submodules or, one could say, finite subsets (the generator sets of them). One way of dealing with local behavior is to consider separation properties. Baer \cite{Bae} introduced separable abelian groups as those in which every finite subset can be `separated'  by (is contained in) a direct summand which is completely decomposable, see \cite{F}. Whatever these completely decomposable groups may be, inspired by Baer's original investigations, a theory of separation has emerged that has taken other classes of groups as the pool from which to choose the separating direct summands, for instance free groups, see \cite{EM}.  

For our purposes the most natural choice of separation module is the class $\Rmod$ of finitely presented modules, as in \cite{GIRT} and here in Definition \ref{sepdef}(3) and (5) (and thereafter). As was shown in \cite[Thm.\,1]{GIRT},  all strict Mittag-Leffler modules are so separable if and only if all pure-projective modules are direct sums of finitely presented modules---while, in general, a module is 
pure-projective if and only if it is a \emph{direct summand} of a direct sum of finitely presented modules. This leads us  to proving that, in general, a module is strict Mittag-Leffler if and only if it is a \emph{direct summand} of a separable module, Corollary \ref{classicalThm.1}. And while this has been known in this form since \cite{HT}, we prove it for arbitrary relativizations in a rather general context, Corollary \ref{f.g.B:1}, as a consequence of a more general result, our first separation theorem \ref{Thm.1}. The other part of this theorem characterizes the (non-strict) relativized Mittag-Leffler modules in terms of another separability, that by pure submodules in place of direct summands. The unrelativized case of it says that a module is Mittag-Leffler if and only if it is a direct summand of a pure separable module, i.e., of a module all of whose finite subsets can be contained in a pure submodule that is finitely presented, Corollary \ref{classicalThm.1}. 

The second separation theorem \ref{Thm.2} describes when the clause `direct summand' is superfluous in Theorem \ref{Thm.1}. The most transparent form this assumes is in the classical, unrelativized case: all strict Mittag-Leffler modules are separable  if and only if all  Mittag-Leffler modules are pure separable---and this is the case precisely when also
in the description of pure-projective modules `direct summand' is superfluous, see Corollary \ref{basicGIRT}.

What are all these relativizations and what the rather more general context? As for the context, the most satisfying results, in  \S6, have been achieved in purely generated classes, in particular, in purely resolving classes (cf.\,Definition \ref{pgen}), as developed in \cite{HR}. Relativizations to classes of modules occurred very naturally in the study of Mittag-Leffler modules in the preceding paper \cite{MLII}. So it is only natural to ask what the corresponding relativizations should be in the strict case. Taking into account the description of (the unrestricted) strict Mittag-Leffler modules, given in \cite[Prop.\,1]{GIRT}, as those in which every tuple freely realizes a certain pp formula, the correct relativization, it seems, should be obtained by relativizing the scope of `freeness' in the free realizations part of the definition, see Definitions \ref{Ffree}(2) and \ref{sKML}. Another approach to achieve a similar goal can be found in \cite{AH} in the shape of `stationary' modules.
 
In the (non-strict) Mittag-Leffler case a fundamental feature was that the relativizations to a certain class depended only on the definable subcategory this class generates, which greatly simplifies many considerations and formulations.  It also allowed for free transfer between formulations for classes $\cK$ of right modules and classes $\cL$ of left modules by simply applying elementary duality to the defining pairs of pp formulas, as was first done in \cite{Her}. So, a module is $\cK$-Mittag-Leffler if and only if it is $\cL$-atomic, whenever $\cK$ and $\cL$ generate elementarily dual definable subcategories, cf.\,\cite[Thm.3.1]{MLII}.

For the strict versions the situation is different, and I found no particular choice of module on the other side naturally lending itself to a similar transfer as in the non-strict case, which is why I mostly stay on the same side as the original module. This results in strict $\cL$-atomic modules having  preference over strict $\cK$-Mittag-Leffler modules. Even more, I \emph{define} the latter concept only for definable subcategories $\cK$---namely, as the strict $\cL$-atomic modules for $\cL$, the definable subcategory elementarily dual to $\cK$, cf.\,Definition  \ref{sKML}.
 
 An actual example showing that the state of affairs is indeed different in the strict case was provided only by an anonymous referee, who pointed out the role of pure-injective modules for this study and how this could be used to obtain two classes of modules that generate the same definable subcategory, yet yield different notions of strict atomicity, see Remark \ref{ref} at the beginning of \S\ref{target}. More on this issue can be found in \S\ref{selfsepsect} and the final  \S\ref{lastsect}.  
 
Some results seem new even in the classical situation of $\cL$ being the entire module category. For instance,
the pp definability of finitely generated endosubmodules of any strict Mittag-Leffler module, Corollary \ref{ppdef}, which  had been previously known for pure-projective modules, cf.\ \cite[Cor.\,2.1.27]{P2}.
Another such result is the passage `from elements to tuples:'   in order for a module to be strict Mittag-Leffler it suffices that every \emph{element} be a free realization of some (unary) pp formula, Corollary \ref{unary}. This is done in the larger context of modules purely generated by a strict relatively atomic module, see Proposition \ref{from1ton}.

In  \cite{HR} axiomatizations of
 purely generated classes by  (usually infinitary) implications of pp formulas were found that provide a better handle on the problems arising with  relativizations. They show that purely generated classes are F-classes in the sense of \cite{PRZ2}.
 In \S\ref{puregen} we pick up  this theory  and specify it to fit purely resolving classes, cf.\,Lemma \ref{pureresolvF}. The resulting more special axiomatization entails  some  downward  L\"owenheim-Skolem-like behavior, called self-separation, see \S\S\ref{ctbleLsep}--\ref{specialF}. A typical application of this is Theorem \ref{LMLisML} showing that relative atomic modules in a purely resolving class are already (fully) Mittag-Leffler.

\bigskip

%In a purely resolving class $\cL$, every strict $\cL$-atomic member is already fully strict atomic, hence strict Mittag-Leffler, Corollary \ref{pureresolvISstrict}. A special case of this is that flat strict $_R\flat$-atomic modules are strict Mittag-Leffler.

 I would like to thank Ivo Herzog and Martin Ziegler for inspiring discussions. Further thanks are due to an anonymous referee for her or his contribution to \S\ref{target} and other helpful comments improving the paper. 
 
% And once again, thanks are due to an anonymous referee---especially for pointing out the role of pure-injectives, see \S\ref{target}.
%  An anonymous referee pointed out that pure-injective modules play an important role, which leads to an example of a class of modules  and the definable subcategory it generates for which strict atomicity 
% The author wishes to thank her or him for this contribution and other helpful comments improving the paper. 
%
\bigskip

{\bf Preliminaries.} The reader is assumed to be somewhat familiar with \cite{MLII}---its main techniques and definitions---in particular, with pp formulas and types, their finite generation and free realization, and definable subcategories and the like.

Throughout, $\cK$ denotes a nonempty class of right $R$-modules and $\cL$ a nonempty class of left $R$-modules. `Module' means left $R$-module unless otherwise specified. $\RInj$ stands for the class of injective (left) $R$-modules, $\sharp$ for the  class of absolutely pure modules (with subscript $R$ to indicate the side when necessary), and $\flat$ for that of flat modules (again, possibly, with a left or right subscript $R$).

 \emph{Map} means  \emph{(homo)\,morphism}, at least when it is  between modules. $(M, N)$ is a shorthand for $\Hom(M, N)$, so  $(M, M)$ means $\End M$. 

We often identify a tuple with the (finite) set of its entries, which will be completely clear from the context.
All tuples and formulas are assumed to be matching.

Herzog \cite[\S4]{Her} discovered that the map that sends a pp pair $\phi/\psi$ to the pair $\D\psi/\D\phi$ can be used to extend elementary duality from pp formulas to closed sets of the Ziegler spectrum  (cf.\ \cite[Thm.5.4.1]{P2}) and hence to definable subcategories, and even to arbitrary theories of modules. Ever since we have the notion of \emph{elementarily dual} definable subcategories that the reader is assumed to be acquainted with (cf.\ \cite[\S3.4.2]{P2}).
Recall from \cite[\S 2.5]{MLII}, that $\D\cK$ denotes the class of all character duals $(K, \mathbb Q/\mathbb Z)$ of $K\in\cK$. The definable subcategory $\langle\D\cK\rangle$ turns out to be elementarily dual to $\langle\cK\rangle$, see \cite{ZZ-H}, \cite{PRZ2} or simply \cite[Cor.3.4.17]{P2}.  

In \cite[Convention 2.7]{MLII} it was stipulated that the classes $\cK$ and $\cL$ were \emph{definably dual} in the sense that the definable subcategories they generate, $\langle\cK\rangle$ and $\langle\cL\rangle$, are elementarily dual.
While no such convention will be made now, as definable subcategories play a different role here, we will freely use definably dual classes in the sense just specified.

\section{Purity}

\subsection{Relativized purity}
The first part of the following definition is discussed in great detail in \cite[\S 5.1]{MLII}. The second is its natural dual, but seems not to have been discussed much. (A stronger version of it occurs in  \cite[\S 5.4]{MLII}.)

\begin{definition}\text{}
 \begin{enumerate}%[\rm (1)]
 \item The map $f: A \to B$ is an \emph{$\cL$-pure monomorphism} if for every tuple $\br a$ in $A$ and every pp formula $\phi$  its image satisfies in $B$, there is a pp formula $\psi\leq_\cL\phi$ \emph{it} satisfies in $A$.
 A submodule $A$ of a module $B$ is \emph{locally $\cL$-pure} if the identical inclusion is.

 \item The map $g: B \to C$ is an \emph{$\cL$-pure epimorphism} if for every tuple $\br c$ in $C$ and every pp formula $\phi$ it satisfies, there is a $g$-preimage $\br b$ in $B$ and a pp formula $\psi\leq_\cL\phi$ \emph{it} satisfies.
\end{enumerate}
\end{definition}

As the partial order does not change when passing from $\cL$ to the definable subcategory it generates, $\cL$-pure and $\langle\cL\rangle$-pure are the same. Beware, $\cL$-pure monomorphisms may not be monomorphisms. E.g., in the category of abelian groups the zero map on the group of two elements is an $\cL$-pure monomorphism for $\cL$, the class of torsionfree abelian groups.

%Is it enough to consider \emph{elements} $c\in C$??? Is, if $f$ and $g$ form a short exact sequence, $f$ $\cL$-pure mono iff $g$ is $\cL$-pure epi????

\begin{rem}\label{LpureInL} If the source of any of these maps is in $\langle\cL\rangle$, the two formulas are equivalent in that source, which yields the usual purity (where $\psi$ can be taken to be $\phi$). In particular,
  \begin{enumerate} [\upshape (a)]
 \item $\cL$-pure submodules that are  in $\langle\cL\rangle$ are pure submodules,
  \item $\cL$-pure epimorphic images of modules  in $\langle\cL\rangle$ are pure epimorphic images.  
 \end{enumerate}
\end{rem}

\subsection{Local splitting}\label{locsplitdef}

Recall  Azumaya's notions,  \cite{Azu}, \cite{Azu2}: a monomorphism $f: N \hookrightarrow M$ is called \emph{locally split} if for every tuple $\br n$ in $N$ there is a \emph{local retract}, i.e., a homomorphism $g_{\br n}: M \to N$ with $g_{\br n}f(\br n)=\br n$. As usual, one calls a submodule $N$ of a module $M$ \emph{locally split}  if the identical embedding is locally split. (Notice, this makes $f$ a monomorphism anyway.) 

For completeness, let me mention that, dually, an epimorphism $g: N\twoheadrightarrow M$ is called \emph{locally split} if for every tuple $\br m$ in $M$ there is a \emph{local section}, i.e., a homomorphism $h_{\br m}: M \to N$ with $gh_{\br m}(\br m)=\br m$. (Again, this makes $g$ an epimorphism anyway.) 

\begin{rem}
 Clearly, in both definitions, the tuples $\br n$ and $f(\br n)$, respectively $\br m$ and $h_{\br m}(\br m)$, have the same pp type. Hence locally split monomorphisms (resp., locally split epimorphisms) are pure.
\end{rem}

It is not surprising that in order for  a monomorphism to be locally split it suffices to check the definition on a generating subset:

%For purposes of reference we single out the easy to verify fact that to check a monomorphism to be locally split, not surprisingly, it suffices to do so on generators.

\begin{rem}
Suppose $N$ is generated by the set $C\subseteq N$ and $f: N\to M$ is a map such that every tuple $\br c$ in $C$ has a local retract, i.e., a homomorphism $g_{\br c}: M \to N$ with $g_{\br c}h(\br c)=\br c$. 
 Then  $f$ is a locally split monomorphism.
\end{rem}

\subsection{Relativized local splitting}
%
%The following concept arises naturally when looking at closure operators for the classes for which a module might be strict Mittag-Leffler etc , see 

Just as we were naturally led to weakenings of purity in \cite{Tf} and \cite{MLII}, we are here led to weakenings of local splitting by the investigations of  how much one can enlarge a class $\cL$ without disturbing the property of being strict $\cL$-atomic (or  strict $\cK$-Mittag-Leffler), see \S\ref{target} below, the only section where this is used.

\begin{definition} Suppose $\cP$ is a class of (left $R$-) modules.

 A homomorphism $f: A \to B$ is \emph{locally $\cP$-split}\/ if for every tuple $\br a$ from $A$ and every tuple $\br p$ from some $P\in\cP$, every map $g: (P, \br p) \to (B, f(\br a))$  factors through $f$ so that $\br p$ goes to $\br a$. More precisely, for every  
such $g: (P, \br p) \to (B, f(\br a))$ there is $h:  (P, \br p) \to (A, \br a)$ such that $g(\br p)= fh(\br p)$.

 A \emph{locally $\cP$-split source} of $B$ is a module $A$ for which  there is such a locally $\cP$-split homomorphism $f: A \to B$.
 A submodule $A$ of a module $B$ is \emph{locally $\cP$-split} if the identical inclusion is.
\end{definition}

\begin{rem} Just as $\RMod$-purity was the usual purity, $\RMod$-locally split submodules are the usual locally split submodules. Further, the ($\Rmod$)-locally split submodules are precisely the pure submodules. Thus, so long as  $\Rmod\subseteq\cP$, all $\cP$-locally split submodules are  pure submodules.

% Just as $\subseteq^{\RMod}_{pure}$  was $\subseteq_{pure}$, here too $\subseteq^{\RMod}_{loc}$  is $\subseteq_{loc}$, as is easily seen.  Further, $\subseteq^{\Rmod}_{loc}$  \emph{is}  purity, so, 
%  as long as $\Rmod\subseteq\cP$,  $\subseteq^\cP_{loc}$ entails the usual purity. ???words???YESSS
\end{rem}

%\subsection{Symbols if needed:} $N\subseteq^\cL_{pure} M$, $N\subseteq_{loc} M$, $N\subseteq^\cP_{loc} M$ $N\subseteq_{pure} M$, $\dot M\models^\cL_{free}$ etc. Drop superscript when it's all of $\RMod$, which is justified by the facts that $\subseteq^{\RMod}_{pure}$  \emph{is} $\subseteq_{pure}$ and $\subseteq^{\RMod}_{loc}$  \emph{is} $\subseteq_{loc}$.
%
%
%\subsection{Examples}
%Investigate $\flat$- and $\sharp$-purity (call them flat and sharp purity???)
%
%

\section{Strict Mittag-Leffler and atomic modules}

\subsection{Free realizations of formulas}
\begin{definition}\label{Ffree} \text{}

\begin{enumerate}    %[\rm (1)]
\item
\label{F-free} \cite[Def.2.2(1)]{Tf}   A \emph{free realization  for $\cL$}  of a pp formula $\phi$, or an  \emph{$\cL$-free realization of $\phi$}, is a pointed module $(M, \bar{m})$ such that  $\bar{m}$ is a (matching) tuple satisfying $\phi$ in $M$ and, whenever a tuple $\bar{c}$ satisfies $\phi$ in a module $L\in \cL$, then there is a map  $M\to L$ %corrected 11/20/18
 sending $\bar{m}$ to $\bar{c}$.\footnote{This differs from the definition made in \cite[\S 1.4]{HR} in that $M$ is not required to be a member of $\cL$. Both differ from Prest's original definition in that the requirement on the underlying module of being finitely presented is also abandoned.}
 \item $M$ is said to be a \emph{strict $\cL$-atomic} if every tuple in $M$ freely realizes some pp formula for $\cL$. %or????: if all of its tuples $\cL$-freely realizes some pp formula. (Clearly, that formula $\cL$-generates the pp type of the tuple in question.)
\item We omit the attribute $\cL$ if it is all of $\RMod$.
\end{enumerate}
\end{definition}

Note that every strict $\cL$-atomic module is $\cL$-atomic in the sense of \cite{habil} (or \cite{MLII})---as it should be. For, if $(M, \bar{m})$ freely realizes $\phi$ for $\cL$ and $\psi\in\pp_M(\br m)$, then, in any $L\in\cL$, if $(L, \bar{a})$ satisfies $\phi$, there is a map $(M, \bar{m})\to (L, \bar{a})$, hence $(L, \bar{a})$ satisfies also $\psi$. Consequently, $\phi$ $\cL$-generates $\pp_M(\br m)$. This property for every $\br m$ in $M$ is what is called $\cL$-atomicity. %, which is, by the main theorem of \cite{habil}, equivalent to $M$ being $\cK$-Mittag-Leffler. 
Now, the main theorem of  \cite{habil} (or \cite{MLII}) says that in case the classes $\cK$ and $\cL$ are definably dual in the sense specified in the preliminaries above (which, recall, means that the definable subcategories they generate  are elementarily dual in the sense of  \cite{Her}), the two properties, $\cL$-atomicity and that of being $\cK$-Mittag-Leffler, are the same---and so we can take it as a definition of $\cK$-Mittag-Leffler. Similarly, \cite[Prop.\,2.1]{GIRT} established that the two properties of strict atomicity (i.e., strict $\RMod$-atomicity in the sense just introduced) and of being strict Mittag-Leffler (in the sense of \cite{RG}) are the same. So it seems only legitimate to extend this to arbitrary definable subcategories and make it a definition, redundant as it may seem. 

\begin{definition}\label{sKML} Given a definable subcategory $\cK$ of right $R$-modules, a module is said to be a \emph{strict $\cK$-Mittag-Leffler} if it is strict $\cL$-atomic for $\cL$, the definable subcategory of left modules elementarily dual to $\cK$ in the sense of Herzog (see Preliminaries). 
%
%
%
%A major decision is due here!!!??? Maybe a general definition can be circumvented. But for elementarily dual  definable subcategories $\cK$ and $\cL$  of $\ModR$ and $\RMod$ resp., this IS the definition.??? Possibly, indeed, the decision can be circumvented and only made for definable subcategories---for all others stay on the right side (which is the LEFT!!!) and use atomicity exclusively??? The problem is that it's hard to know (at this point???) which definition of $\D\cK$ to champion when $\cK$ is no longer definable.???
%
%Or should we take characters $K^+$ for $\D K$??? Then we would need a theorem ??? that strict $\D\cK$-atomicity implies strict $\langle\D\cK\rangle$-atomicity for that particular choice of $\D K$.
\end{definition}

We will still prefer the terms (strict) $\cL$-atomic over their Mittag-Leffler version---for three reasons: first it allows us to stay on the same side of the ring, second, and more importantly, it makes available our main techniques involving pp formulas, third, and most importantly, it depends here on how we apply elementary duality to a class of modules, as the concepts may no longer be invariant under definable subcategories.  At the same time though, we want to keep in mind  their connotations `on the other side:' as strict $\cK$-Mittag-Leffler modules---provided $\cK$ and $\cL$ are definable subcategories that are elementarily dual.

\begin{rem}
 The strict Mittag-Leffler modules of \cite{RG} are  the strict $\cK$-Mittag-Leffler modules for $\cK = \ModR$ (i.e.,  the strictly $\cL$-atomic modules for $\cL = \RMod$), \cite[Prop.\,2.1]{GIRT}, or, equivalently,   \cite{Azu2}, the locally pure-projective modules, i.e., the modules such that every pure epimorphism onto them locally splits in the sense of \S\ref{locsplitdef}. 
\end{rem}

It is not hard to see that within a definable subcategory $\cD$, for a $\cD$-atomic module to have the strict property it suffices to check it for single modules in $\cD$ one by one:

\begin{lem}\label{onebyone}
 Let $M$ be an $\cL$-atomic module and $\cL'=\bigcup_{i\in I} \cL_i\subseteq\langle\cL\rangle$.
 
 If $M$ is strict $\cL_i$-atomic for every $i\in I$, then it is strict $\cL'$-atomic.
\end{lem}
\begin{proof}
 By hypothesis, every tuple $\br a$ in $M$ satisfies a pp formula $\phi$ that $\cL$-generates---hence also $\langle\cL\rangle$-generates---its type $\pp_M(\br a)$, and there are $\phi_i$ in that type that $\br a$ freely realizes for $\cL_i$, $i\in I$. As mentioned before, this implies that $\phi_i\leq_{\cL_i}\phi$. But $\cL_i\subseteq\langle\cL\rangle$ implies also $\phi\leq_{\cL_i}\phi_i$, which shows that we may take $\phi$ for all the $\phi_i$. Consequently, $(M, \br a)$ freely realizes $\phi$ for $\cL'$, as required.
\end{proof}

\smallskip
\subsection{Pure submodules of atomic modules}
It is clear from the definition that the class of strict $\cL$-atomic modules is closed under arbitrary direct sum and pure submodules. Here is a refinement of the latter.

\begin{lem}\label{Lpureofstrict} If $f: N \to M$ is an $\cL$-pure monomorphism with $M$ strict $\cL$-atomic, then $N$ is strict $\cL$-atomic. 

In particular,
$\cL$-pure submodules of strict $\cL$-atomic modules are strict $\cL$-atomic. 
(Without `strict' this is \cite[Cor.\,6.3]{MLII}.)
\end{lem}
\begin{proof} For notational simplicity, assume $N$ is an $\cL$-pure submodule of the  strict $\cL$-atomic module $M$ and $f$ is the identical inclusion.
 Let $\br n$ be a tuple in   $N$. Then $(M, f(\br n))$ freely realizes some formula $\phi$ for $\cL$ which is $\cL$-equivalent to some $\psi$ in $\pp_N(\br n)$. To see that $(N,\br n)$ freely realizes  $\psi$ for $\cL$, consider $L\in\cL$ and a tuple $\br c$ therein that satisfies $\psi$. By $\cL$-equivalence, it also satisfies $\phi$, whence we have a map $(M, f(\br n))\to (L, \br c)$. This yields  the desired map $(N, \br n)\to (L, \br c)$.
\end{proof}

\begin{lem}\label{locsplitsub}
 Suppose $f: N \to M$ is an $\cL$-pure monomorphism with $M$ strict $\cL$-atomic.
 If $N$ is itself in $\cL$, then $f$ is a locally split (and hence pure\footnote{It's pure anyway, see Remark \ref{LpureInL}.}) monomorphism. If $N$ is, in addition, finitely generated, $f$ splits.
 
 In particular,
$\cL$-pure submodules of strict $\cL$-atomic modules that are themselves in $\cL$ are locally split (and hence pure)
 submodules. If they are, in addition, finitely generated, they are even direct summands.
 \end{lem}
\begin{proof}
%(We know already that it's pure from Remark \ref{Lpure}.) 
Let $\br n$ be any tuple in $N$ and let $(M, f(\br n))$ $\cL$-freely realize the pp formula $\phi$.
 As  $N\in\cL$, by $\cL$-purity, $\phi$ is contained in the type of  $\br n$ in $N$, and so we have $(M, \br n)\to (N, \br n)$, which is the desired local  retract for ${\br n}$. \end{proof}

\subsection{Closure of the target class}\label{target} First a crucial observation about pure-injectives provided by the referee, which culminates in an example showing that, in contrast to $\cL$-atomicity (and $\cK$-Mittag-Leffler modules), which depends only on the definable subcategory $\cL$ (or $\cK$) generates, strict $\cL$-atomicity is, in general, different from $\langle\cL\rangle$-atomicity.

\begin{rem}\label{ref}\text{}

\begin{enumerate}[\upshape (a)]
 \item If $\cL$ consists of pure-injective modules only, $\cL$-atomic implies strict $\cL$-atomic. 
 For, if the pp type $\pp_M(\br a)$ in $M$ is $\cL$-generated by $\phi$ which $\br b$ in $L\in\cL$ satisfies, then $\pp_M(\br a)\subseteq\pp_L(\br b)$, hence pure injectivity of $L$ guarantees a map $(M, \br a) \to (L, \br b)$, as desired (see e.g.\ \cite[Thm.4.3.9]{P2}).
 \item If $\cL\subseteq\cL'\subseteq\langle\cL\rangle$ are such that $\cL'\setminus \cL$ consists of pure-injectives only, then every strict $\cL$-atomic module is  strict $\cL'$-atomic. 
 
 This follows from the implications 
strict $\cL$-atomic  $\Longrightarrow$ $\cL$-atomic $\Longrightarrow$ $\langle\cL\rangle$-atomic $\Longrightarrow$ $\cL'\setminus \cL$-atomic $\Longrightarrow$ strict $\cL'\setminus \cL$-atomic (by (1)), which, together with strict $\cL$-atomicity,  yields strict $\cL'$-atomicity by Lemma \ref{onebyone}.

\item
Let $\cD$ be a definable subcategory and $\cL$ be the class of all of pure-injectives in $\cD$. If strict $\cL$-atomic implies strict $\langle\cL\rangle$-atomic, then $\cD$-atomic implies strict $\cD$-atomic.

This follows from the implications 
$\cD$-atomic  $\Longrightarrow$ $\cL$-atomic $\Longrightarrow$ strict $\cL$-atomic $\Longrightarrow$ strict $\langle\cL\rangle$-atomic $\Longleftrightarrow$ strict $\cD$-atomic, the second of which is  (1); for the last notice $\langle\cL\rangle=\cD$, for any module is in the same definable subcategories as is its pure-injective hull.

\item 
Consequently, every definable subcategory $\cD$ with a $\cD$-atomic member which is not strictly $\cD$-atomic gives rise to a subclass $\cL$ (namely, the class of all pure-injective members of $\cD$) such that strict $\langle\cL\rangle$-atomicity is stronger than strict $\cL$-atomicity.

\item As a concrete example, take $\cL$ to be the class of all pure-injective abelian groups. The definable subcategory it generates is the category of all abelian groups. Since there are Mittag-Leffler abelian groups that are not strict Mittag-Leffler, e.g.,
Example  \ref{counterex} below, not every strict $\cL$-atomic abelian group is strict $\langle\cL\rangle$-atomic (= strict Mittag-Leffler).
\end{enumerate}
\end{rem}

The obstacle in extending strict $\cL$-atomicity to strict $\langle\cL\rangle$-atomicity  is the closure of $\cL$ under pure submodules (note, \emph{every} module is pure in its pure-injective envelope). If we strengthen purity to a certain weak local splitting, however, we can prove this. Again, the additional clause about pure-injectives is due to the referee.

\begin{prop} Suppose $\cL$ is any class of modules and  $\cP$ is the class of strict $\cL$-atomic modules.  Let $\cL'$ be the union of $\cL$ and the class of all pure-injectives in $\langle\cL\rangle$, and let $\bar{\cL}$ be the  closure of $\cL'$ under  direct product, direct limit and local $\cP$-split sources (in particular, under local $\cP$-split submodules).

Then every strict $\cL$-atomic module is  strict $\bar{\cL}$-atomic.
\end{prop}
\begin{proof} Suppose $M$ is strict $\cL$-atomic, hence, by (2) of the previous remark, also strict $\cL'$-atomic. That this condition is preserved in direct products of the target modules follows easily from the fact that pp formulas are p-functors, i.e., commute with direct product. Namely, let $\br m \in\phi(M)$ freely realize a pp formula $\phi$ for $\cL'$. If $N$ is the product of some $L_i\in\cL'$ and a tuple $\br n$ satisfies $\phi$ in $N$, then each coordinate of $\br n$ in $L_i$, call it $\br l_i$, satisfies $\phi$.  So there are maps $f_i: (M, \br m) \to (L_i, \br l_i)$,  whose product sends $(M, \br m)$ to $(N, \br n)$, as desired.

To show that this condition is preserved in  direct limits, suppose  the $L_i$ form a directed system with direct limit $N$ and the tuple $\br n$ satisfies  $\phi$ in $N$.  This is possible only if $\phi$ was true along some tail of $\br n$ in the direct system. In particular, there is a preimage $\br l_i$ of $\br n$, under a canonical map, which satisfies $\phi$. By hypothesis, there is a map $(M, \br m)\to (L_i, \br l_i)$, which, composed with that canonical map, yields a map $(M, \br m)\to (N, \br n)$, as desired.

Finally, suppose $f: N\to L$ is  locally $\cP$-split, $L\in\cL'$, and $(N, \br n)$ satisfies $\phi$. Being existential, $\phi$ is also satisfied in $(L, f(\br n))$. By hypothesis on $M$, there is a map $(M, \br m)\to (L, f(\br n))$, which, by hypothesis on $N$, factors through $f$, and thus yields the desired map to $(N, \br n)$.
\end{proof}

A simple argument using $\Rmod\subseteq\cP$  shows that $\langle\cL\rangle$ is closed under local $\cP$-split sources. The other operations being part of the definition of definable subcategory, one infers that
 $\bar{\cL}$ is contained in the definable subcategory $\langle\cL\rangle$. It might be short of being the whole thing by leaving out some pure submodules that are not locally $\cP$-split. In fact, we have the following.

\begin{lem}
 Suppose $A$ is a pure submodule of some $L\in\cL$ such that every strict $\cL$-atomic module is also strict $A$-atomic.
 
 Then $A$ is locally $\cP$-split in $L$ (with $\cP$, the class of strict $\cL$-atomic modules).
\end{lem}
\begin{proof}
 Suppose  $A\subseteq L$ and $g$ is a map from $P\in\cP$ to $L$ with $g(\br p)=\br a$ for given tuples $\br a$ in $A$ and $\br p$ in $P$. As $P\in\cP$, the tuple $\br p$ freely realizes some pp formula $\phi$ for $\cL$  in $P$. As $g$ preserves pp formulas, $\br a$ satisfies $\phi$ in $L$, and, by purity, also in $A$.
 
  Since $P$ is strict $\cL$-atomic, by hypothesis, it is also strict  $A$-atomic, so $(P, \br p)$ must freely realize some pp formula $\psi$ for $A$. Then $\phi\leq_\cL\psi$, and as $A$ is in $\langle\cL\rangle$, the tuple $\br a$ satisfies $\psi$ in $A$, which yields the desired map $(P, \br p)\to (A, \br a)$.
 \end{proof}

We will return to this issue in \S\S\ref{ctbleLsep} and \ref{lastsect}.

%\subsubsection{Free realizations of sequences of formulas}???:this should be shortened to just free realizations of formulas and the strict lemma and probably incorporated somewhere earlier???
%
% Passing from single formulas to sequences of such we arrive at the following 
%
%
%
%\begin{definition}
%Let $\Phi$ be a sequence of pp formulas of growing length, $\phi_i=\phi_i(x_0, \ldots, x_i)$ $(i<\omega)$. (Here the $x_i$ themselves are allowed to be tuples.???)
%
%A \emph{free realization of  $\Phi$ for $\cL$} is a module  $A$ pointed  by a sequence of matching tuples $a_0, a_1, \ldots$  $(i<\omega)$ that realize the formulas from $\Phi$ in $A$ and such that whenever  a countable sequence $b_0, b_1, \ldots$ of matching tuples in a module $L\in \cL$ satisfies $\Phi$ (i.e., for all $i<\omega$, the (concatenated???) tuple $(b_0, \ldots, b_i)$ satisfies $\phi_i$ for all $i$), then there is a map $A\to L$ sending $a_i$ to $b_i$ for all $i$.
%\end{definition}
%
%\begin{lem} \cite[Lemma 2.3]{Tf}\label{freerealsequ}Needed at all in this paper???
%  Suppose  $A$ is an $\cL$-atomic module generated by $a_0, a_1, \ldots$. Let, for all $i$, the pp type of $(a_0, \ldots, a_i)$ be $\cL$-generated by the pp formula $\phi_i=\phi_i(x_0, \ldots, x_i)$.\footnote{Note, this alone suffices to make $A$ an $\cL$-atomic module.} Then $A$ is a free realization of the sequence $\phi_0, \phi_1, \ldots$ for $\cL$.
%\end{lem}

\subsection{Countably generated atomic modules are strict}
 \cite[Prop.\,2.4]{Tf} says that every tuple in a countably generated $\cL$-atomic module  freely realizes some pp formula for $\cL$.  This is so basic a fact for this topic that  
 we restate it as follows and make use of it without much mention.

\begin{lem}[The `Strict' Lemma]\label{StrictLemma}\cite[Prop.\,2.4]{Tf} %Really check this again WELL???

 Countably generated $\cL$-atomic modules are strict $\cL$-atomic. Consequently, countably generated strict $\cL$-atomic modules are  strict $\langle\cL\rangle$-atomic.
\end{lem}
\begin{proof}
 The first part of the statement is \cite[Prop.\,2.4]{Tf}. Invoking \cite[Cor.\,3.6(a)]{MLII} saying that $\cL$-atomic is the same as $\langle\cL\rangle$-atomic, the second follows at once.
\end{proof}

\subsection{Traces}
The role of traces in this context has been made clear in
\cite{Z-H} and \cite{Gar}, see also \cite[Cor.\,1.2.17]{P2}. Given a tuple  $\br m$  in a module $M$, by the \emph{trace} of $\br m$ in a module $N$ we mean the set $(M, N)\,\br m := \{h(\br m)\, | \, h\in (M, N)\}$. By an \emph{$\cL$-trace} we mean a trace in a module from $\cL$. 

\begin{lem}\label{traces} 
$(M, \bar{m})$ is a free realization  for $\cL$ of the pp formula $\phi$ if and only if $\br m$ satisfies $\phi$ in $M$ and 
$\phi$ defines all $\cL$-traces of\/ $\br m$, i.e., $(M, L)\,\br m = \phi(L)$ for all $L\in\cL$.
\end{lem}
\begin{proof} As morphisms preserve pp formulas, $\br m\in\phi(M)$ implies $(M, F)\,\br m \subseteq \phi(F)$, for any module $F$. The inverse inclusion holds if and only if $\br m$ can be mapped to every element of $\phi(F)$. 
\end{proof}

Note, if $M\in\cL$ then $\br m$ is in its own $\cL$-trace  and therefore automatically satisfies any pp formula that defines all its $\cL$-traces.

\begin{prop}\label{K-sML} %Let $\cK$ and $\cL$ be mutually dual definable subcategories of right and left $R$-modules, respectively. 
The following are equivalent for any (left $R$-) module $M$.

%Either do this here without "elements only" or put "from tuples to elements (or v.v.) earlier", which seems better???
\begin{enumerate} [\upshape (i)]
\item $M$ is strict $\cL$-atomic.
\item Every tuple $\br m$ in $M$ satisfies a pp formula that defines all traces of $\br m$ in modules from $\cL$.% (in any module).
%\item Every tuple  $\br m$ in $M$ is contained in a pp subgroup $\phi(M)$ such that there is a free realization $(N, \br n)$ of $\phi$ with $\phi(N)=(M, N)(\br m)$. %in which $\phi$ defines the trace of $\br m$, i.e.
%\item The same for elements $m\in M$ only.
\item Every tuple  $\br m$ in $M$ satisfies a pp formula $\phi$ that is freely realized in some module $N$ %$(N, \br n)$
with $\phi(N)=(M, N)\,\br m$. %in which $\phi$ defines the trace of $\br m$, i.e.
%\item The same for elements $m\in M$ only.
\end{enumerate}
\end{prop}
\begin{proof} (i) and (ii) are equivalent by the lemma. Since every pp formula has a (total) free realization somewhere (in some module $N$, even a finitely presented one), (iii) is equivalent as well.
%(i) $\Rightarrow$ (ii). Let $(M, \br m)$ freely realize $\phi$ and $N$ be any module. As homomorphisms preserve pp formulas, the trace of $\br m$ in $N$ is contained in $\phi(N)$. But $\br m$ realizes $\phi$ freely and can therefore be mapped onto any tuple in   $\phi(N)$. In other words, $(M, N)(\br m)$ must be all of $\phi(N)$.
%
%More???
\end{proof}

This allows us to extend a result known for pure-projective modules to strict Mittag-Leffler modules, cf.\ \cite[Cor.\,2.1.27]{P2}.

\begin{cor}\label{ppdef}%\cite[???]{GIRT}
 Every finitely generated endosubmodule of a strict Mittag-Leffler  module is pp-definable.  
% 
% Relativize!!!???
%
\end{cor}
 \begin{proof}
 Let $M$ be strict Mittag-Leffler, that is, strict ($\RMod$-) atomic. An endosubmodule generated by elements $m_i\in M$ ($i\in I$) is the sum of all the traces $(\End M)\,m_i = (M, M)\,m_i$. These being pp-definable, their sum is too (so long as $I$ is finite).
\end{proof}

Clearly, the same is true for finite sums of traces of $k$-tuples $(M, M)\,\br m_i$, so that also finitely generated submodules of $M^k$---regarded as a module under the diagonal action of $\End M$ again---are pp-definable (in $M$). 

\begin{rem}
 Every finitely generated endosubmodule of a module $M$ is pp-definable if and only if $M$ is strict $M$-atomic, for the finitely generated endosubmodules of $M$ are precisely the traces under $\End M = (M, M)$ of the generator tuples.
\end{rem}

\subsection{Strict $_R\sharp$-atomic modules} Over left coherent rings, where $\flat_R$ and $_R\sharp$ form definable subcategories,  \cite[Thm.3.4.24]{P2}, which are elementarily dual, these are, by definition, the strict $\flat_R$-Mittag-Leffler modules. However, we prefer to work with $_R\sharp$-atomic modules, which make sense over any ring. (So does the duality of $\flat_R$ and $_R\sharp$, namely, Herzog showed that these classes are what we call definably dual, i.e., they generate elementarily dual subcategories \cite[\S12]{Her}, cf.\ \cite[Prop.3.4.26]{P2}.)

It has been known since \cite{Goo} that all modules are $R_R$-Mittag-Leffler (equivalently, $\flat_R$-Mittag-Leffler, or $_R\sharp$-atomic)  if and only if $R$ is left noetherian, cf.\ also \cite[Cor.\,4.3]{MLII}. In our terminology, $R$ is left noetherian if and only if every module is $_R\sharp$-atomic. We are going to strengthen this by showing that then all modules are even strict $_R\sharp$-atomic. First a general observation, whose original proof became obsolete with the referee's Remark \ref{ref}(1) above.

\begin{rem}
 $R_R$-Mittag-Leffler (= $_R\sharp$-atomic) modules are strict $\RInj$-atomic.
\end{rem}
%\begin{proof}
%In \cite{Goo} it was proved that $M$ is $R_R$-Mittag-Leffler iff the inclusion of every finitely generated submodule of $M$ into $M$ factors through a finitely presented module, see also \cite[Cor.\,2.7]{habil}.%, where this was reproved in the following style.
%
%Let $\br m$ be any tuple in $M$ and assume that for the submodule $A$ generated by $\br m$ we have the following factoring $(A, \br m)\stackrel{f}\to(P, \br c) \to (M, \br m)$ of pointed modules of the inclusion $A\subseteq M$. with $P$ finitely presented. Then the pp formula $\phi$ that $(P, \br c)$ freely realizes generates the quantifier-free type of $\br m$ in $M$ and it also $\sharp_R$-generates the entire pp type  of $\br m$ in $M$, \cite[Cor.\,4.3]{MLII}. We claim, $(M, \br m)$ freely realizes $\phi$  for $\RInj$.
%
% So let $E$ be injective and $(E, \br e)$ realize $\phi$. We need to find a map $(M, \br m)\to (E, \br e)$. There is a map $(P, \br c)\stackrel{g}\to (E, \br e)$ (whether $E$ is injective or not), simply because $(P, \br c)$ is a free realization of $\phi$. By injectivity of $E$ now, $gf$ factors through the inclusion  $(A, \br m)\subseteq (M, \br m)$, which yields the desired map $(M, \br m)\to (E, \br e)$.
% \end{proof}

\begin{prop}
 A ring $R$ is  left noetherian if and only if  every module is  strict $_R\sharp$-atomic.
 
  In particular, every abelian group has this property.
\end{prop}
\begin{proof}
% Since, over a left coherent rings, $\flat_R$ and $_R\sharp$ are dual definable subcategories, strict $\flat_R$-Mittag-Leffler is, by definition, here strict $_R\sharp$-atomicity. But i
% 
 In the noetherian case, $_R\sharp=\RInj$, so the implication from left to right follows from the remark (and the fact that then every module is $_R\sharp$-atomic). As mentioned, for the  converse it suffices to have all modules (plain) $_R\sharp$-atomic (= $R_R$-Mittag-Leffler).
\end{proof}

\begin{exam}
 By \cite[Prop.\,7]{AF}, an abelian group $M$ is Mittag-Leffler iff it has trivial first Ulm group (i.e., the intersection of all $nM$ is $0$) and $M/\tor M$ is Mittag-Leffler. Hence no divisible group is Mittag-Leffler. But all of them are strict $_R\sharp$-atomic.
\end{exam}

%Interesting:??? We need the result that every module over any ring has an
%absolutely pure preenvelope (see D. Adams [1])???
%
%Adams, D.D., Absolutely pure modules, Ph.D. Thesis, University of Kentucky, 1978.???
%
%\bigskip
%
%
%Does this need a separate subsection????:
% $R_R\in\langle\cK\rangle$ or, equivalently, $_R\sharp\subseteq\langle\cL\rangle$.
% 
% \bigskip
%
%
%\smallskip
\subsection{Strict $_R\flat$-atomic modules} Dually to what was said at the beginning of the previous section, now it is the right coherent rings over which these are the strict $\sharp_R$-Mittag-Leffler modules. But again, we wish to work with arbitrary rings and stick to  strict $_R\flat$-atomic modules instead. 

\begin{exam} (Of strict $_R\flat$-atomic abelian groups that are not Mittag-Leffler.)
 By \cite[Prop.\,5.1]{Tf}, every torsion abelian group is $_R\flat$-atomic---in fact, by \cite[Thm.\,6.8 or Cor.\,6.12]{Tf}, every group $M$ with $M/\tor(M)$ Mittag-Leffler is.
 Hence every countable such group is even strict $_R\flat$-atomic, cf.\ Lemma \ref{StrictLemma}. Now one gets a host of such groups which are not Mittag-Leffler from \cite[Prop.\,7]{AF}, see the previous Example. Namely,   take any countable torsion group with non-trivial first Ulm group, like the Pr\"ufer groups.
\end{exam}

In contrast, flat strict $_R\flat$-atomic modules turn out to be always strict atomic (i.e., strict Mittag-Leffler), as will be shown in Corollary \ref{flatISstrict} below. More  will be said in the introduction to \S\ref{flatcase} as a consequence  of some separation results for purely generated classes.

\section{Pure generation}\label{puregen}
 
\begin{definition}\cite[\S 2]{HR}\label{pgen}
\begin{enumerate}%[\rm (a)]
 \item  A module  is \emph{purely generated} by a class $\cB$ if it is a pure epimorphic image of a direct sum of modules from $\cB$. A class is  \emph{purely generated by  $\cB$} if every module in it is. The class of \emph{all}\/ modules purely generated by  $\cB$ is denoted $\PGen\cB$.
 \item Following \cite[\S 2]{HR},  we let $\cC=\Add\cB$  and  $\tC$ be the class of all pure epimorphic images of modules from $\cC$ (equivalently, from $\oplus\cB$). So  $\tC$ is a shorthand for %i.e., the class purely generated by $\cB$, which  is therefore also denoted 
 $\PGen\cB$.

\item  \cite[Rem.\,2.4(a)]{HR}. A class is \emph{purely resolving} if it is purely generated by  pure-projective modules. 
\item For the ease of exposition, we always assume $\cB$ to be closed under finite direct sum.
\end{enumerate}
\end{definition}

Note, a class is pure resolving whenever it is purely generated by  \emph{locally} pure-projective  modules, \cite[Cor.\,2.5]{HR}, which means we may allow $\cB$ to consist of strict Mittag-Leffler modules and still have $\tC=\PGen\cB$ purely resolving.

\subsection{Basic properties}

\begin{prop}
Suppose $M$ is a pure-epimorphic image of a strict $\cG$-atomic module $L\in \cL$.
 
 If a tuple $\br m$ in $M$ freely realizes some pp formula for $\cL$, then it  freely realizes also some (possibly different) pp formula for $\cG$.
 
 Consequently, if $M$ is strict $\cL$-atomic,  it is also  strict $\cG$-atomic.
\end{prop}
\begin{proof}
 Consider a pure epimorphism $h: L \to M$. Let $\br m$ (in $M$) freely realizes a pp formula $\phi$ for $\cL$ and choose a preimage $\br l\in\phi(L)$ of $\br m$ by purity. By hypothesis on $\phi$, there is a map $f: (M,  \br m)\to (L, \br l)$. Being strict $\cG$-atomic, $(L, \br l)$ freely realizes a pp formula $\psi$ for $\cG$. Application of $h$ shows that $(M,  \br m)$ also satisfies $\psi$. To see that it does so freely for $\cG$, let $G\in\cG$ and $\br g$ be a tuple that satisfies $\psi$ in $G$. Then there is a map $(L,  \br l) \to (G, \br g)$. Composed with $f$, this yields the desired map $(M,  \br m) \to (G, \br g)$.
\end{proof}

\begin{cor}
Let $\cL$ be a class of strict $\cG$-atomic modules and $\tC$  a class purely generated by  $\cL$, \cite[\S 2]{HR}. 

Then every strict  $\cL$-atomic member of $\tC$  is   strict $\cG$-atomic.
\end{cor}

Special cases are:

\begin{cor}\label{pureresolvISstrict}
 If $\cL$ is a purely resolving class, %i.e., purely generated by (locally) pure-projectives,  \cite[Rem.\,2.4(a), Cor.\,2.5]{HR},  
 then every strict $\cL$-atomic member of $\cL$ is strict atomic, i.e., strict Mittag-Leffler.\qed
\end{cor}

\begin{cor}\label{flatISstrict}
Every flat strict $_R\flat$-atomic module is strict atomic, hence strict Mittag-Leffler.
\end{cor}
\begin{proof}
The class $_R\flat$ of flat modules is purely resolving, in fact, purely generated by the finitely generated projectives, cf.\, \cite [before Thm.2.1]{HR}. Now the previous corollary applies.
 \end{proof}
 
 Compare this to \cite[Thm.\,3.10]{Tf}, which says the same without `strict.'

\subsection{From elements to tuples}

It is always desirable to reduce a condition on tuples  to the same condition on just elements. 

\begin{prop}\label{from1ton} Suppose $M$ is a pure-epimorphic image of a strict $\cL$-atomic module in $\cL$.
 Then the following are equivalent.
 
% The following are equivalent for any module $M$ that is a pure-epimorphic image of a strict $\cL$-atomic module contained in $\cL$. 

 \begin{enumerate}[\upshape (i)]
%\item $M$ is sML.
\item $M$ is strict $\cL$-atomic.
    \item Every \emph{element}  in $M$ freely realizes a certain pp formula for $\cL$.
\end{enumerate}
\end{prop}
\begin{proof}
 We prove the nontrivial direction by induction on the length of the tuple. Suppose $\br m_0$ is a tuple in $M$ $\cL$-freely realizing  a pp formula $\phi_0$ and $m\in M$. We are going to find a pp formula $\phi$ that the concatenated tuple $(\br m_0, m)$ freely realizes  for $\cL$ (in $M$). %(And, of course, we may take $\br m_1$ of length $1$.)
 
By hypothesis, there is a pure epimorphism $g$ from a strict $\cL$-atomic module $L\in\cL$  onto $M$. Using purity of $g$, choose a preimage $\br b_0$ of $\br m_0$ in $L$ satisfying $\phi_0$. By choice of  $\phi_0$, there's a map $f_0: (M, \br m_0)\to (L, \br b_0)$. Note, $gf_0(\br m_0)=\br m_0$. 

By (ii), there is a pp formula $\phi'$ that the element $m':= m - gf_0(m)$ freely realizes   for $\cL$ in $M$, and, by purity of $g$ again, a preimage $b'$ in $L$ realizing it. Then there is a map $f': (M, m')\to (L, b')$. Again, $gf'(m')=m'$.

Let $f$ be the map $f' + f_0 - f'gf_0$. An easy calculation shows that $f(\br m_0)=\br b_0$ and $f(m) =  f'(m')+ f_0(m)$, and another one that therefore $gf$ fixes both $\br m_0$ and $m$. Thus $(\br m_0, m)$ and $f(\br m_0, m)$ have the same pp type (in $M$ and $L$, resp.).

%
%
%Then $f(\br m_0)=f'(\br m_0) + f_0(\br m_0) - f'gf_0(\br m_0)= f'(\br m_0) + f_0(\br m_0) - f'(\br m_0)=\br p_0$. Further, $f(m) = f'(m) + f_0(m) - f'gf_0(m) = f'(m')+ f_0(m)$ and hence $gf(m)= gf'(m')+ gf_0(m)=m' + gf_0(m)=m$. 
%
%This shows that $gf(\br m_0, m)=(\br m_0, m)$ and therefore $(\br m_0, m)$ and $f(\br m_0, m)$ have the same pp type (in $M$ and $P$, resp.).

It remains to note that, by choice of $L$,  the latter type is generated by a single pp formula, say $\phi=\phi(\br x_0, x)$,  that $(L, \br b_0, f(m))$ freely realizes for $\cL$. Composing with the map $f$, we finally conclude that $(M, \br m_0, m)$ freely realizes that same $\phi$  for $\cL$, as desired.
\end{proof}

Warfield showed every module is a pure-epimorphic image of a pure-projective module (one can, in fact, take a direct sum of finitely presented modules), \cite[33.5]{W}.  Hence the lemma holds for $\cL=\RMod$. 

\begin{cor}\label{unary}
 A module is strict Mittag-Leffler if and only if every \emph{element} freely realizes some (unary) pp formula.\qed
\end{cor}

%whence a module is strict atomic provided every \emph{element} freely realizes a certain pp formula. More can be said. How???
%
%???
%
%This passage from elements to tuples can be made in many instances in this paper, for instance, in local (or finte???) splitness (\cite{Azu}), but also in separability results like ???? (cite all results from this paper).
%
%Let me reiterate, this holds for  $\cL=\RMod$.

\subsection{Axiomatizability}\label{Ax} Consider  \emph{F-classes} as introduced in \cite{PRZ2}, which are classes of modules that can be axiomatized by \emph{F-sentences}, i.e., implications of the form  $\phi \to \bigvee \phi_i$, where $\phi$ and the $\phi_i$ are pp formulas. Here the conclusio is allowed to be an  infinitary disjunction, but such that the $\phi_i$ are closed under finite sum so that this implication is equivalent to the implication $\phi \to \sum \phi_i$, cf.\,\cite [\S 2]{HR}. The `F' comes from flat, for the class of flat modules is an F-class, see next section.

\cite[Thm.\,2.1]{HR} exhibits an axiomatization of purely resolving classes by F-sentences. To describe the specific axioms, consider such a class, $\tC=\PGen\cB$. Let $\ppf_\cB$ stand for the set of pp formulas that are freely realized (for \emph{all} of $\RMod$) by a tuple in some module from  $\cB$. Given a pp formula $\phi$, let $\ppf_\cB\phi$ be the set of all formulas (of same arity) in $\ppf_\cB$ that are below $\phi$, i.e.,  $\ppf_\cB\phi =\{\theta\in \ppf_\cB\,\, | \,\, \theta \leq\phi\}$.  The axioms for $\tC$ are now all F-sentences of the form $\phi \to \bigvee \ppf_\cB\phi$, where $\phi$ runs over all pp formulas (it suffices to consider $1$-place pp formulas).

The overall assumption that $\cB$ be closed under finite direct sum guarantees that $\ppf_\cB$ (hence also $\ppf_\cB\phi$) is closed under (finite) sums of pp formulas, so that the disjunction in the conclusio can, again, be replaced by the sum. In other words, the above axiom is equivalent to  $\phi \to \sum \ppf_\cB\phi$.

\begin{rem}
 One can generalize this axiomatization result to classes that are purely generated by relativized strict atomic modules, however those will no longer have the special features singled out in the next section.
\end{rem}

\subsection{Special F-classes}\label{specialFdef}
In general, one may take $\phi_i\wedge \phi$ in the conclusio of an F-sentence, to ensure that the implication's converse is automatically true: then the implication $\phi \to \bigvee \phi_i$ entails the equivalence of $\phi$ and $\bigvee \phi_i$. 

(Beware, this may lead outside the realm of a specific shape of formula. For instance, the conjunction of a formula from $\ppf_\cB$ with an arbitrary pp formula $\phi$ may no longer be in $\ppf_\cB$. This is the reason, why,  in \S\ref{Ax}, we worked with $\ppf_\cB\phi$ instead, cf.\ the proof of Lemma \ref{pureresolvF} below.)

We will see shortly that  $_R\flat$ is so axiomatized. But the axioms for the class of  flat modules have yet another special feature. Given another such axiom, $\psi \to \bigvee \psi_i$ with $\psi_i\leq\psi$, if every disjunct $\phi_i$ of $\phi$ is equivalent to one of the $\psi_i$, then $\phi \to \bigvee \phi_i$ implies $\phi\to\psi$. Plainly speaking, one disjunction (trivially) implies another disjunction provided all disjuncts of the former are disjuncts of the latter. The special feature we are after now is when this is the only way that the former disjunction can imply the latter. Here is the precise definition.

\begin{definition}\text{}

\begin{enumerate}
 \item An F-sentence $\phi \to \bigvee \phi_i$ is in \emph{standard form} if $\phi_i\leq \phi$ for every $i$.
 \item Given F-sentences in standard form,  $\phi \to \bigvee \phi_i$ and $\psi \to \bigvee \psi_j$, we say \emph{$\phi$ trivially implies $\psi$ (in $\cF$)} if every disjunct $\phi_i$ (of $\phi$) (logically) implies some disjunct $\psi_j$ (of $\psi$), i.e., for all $i$ there is a $j$ with $\phi_i\leq\psi_j$.
 \item  An F-class $\cF$ is \emph{special} if  it has a set of F-axioms $\phi \to \bigvee \phi_i$ in standard form, one for each pp formula $\phi$ (this is important!), and  such that for every two axioms $\phi \to \bigvee \phi_i$ and $\psi \to \bigvee \psi_j$ from that set, we have $\phi \leq_\cF \psi$ (if and) only if $\phi$ trivially implies $\psi$. %in $\cF$.
 \end{enumerate}
 
 The standard form of $\phi \to \bigvee \phi_i$ guarantees that $\phi$ is $\cF$-equivalent to the disjunction of the $\phi_i$, and similarly for $\psi$. %(Clearly, one can always bring an F-sentence into (logically equivalent) standard form by concatenating every disjunct of the conclusio with the premise.) 
 Hence, if every disjunct of $\phi$  implies some disjunct $\psi_j$ in the above $\phi$ and $\psi$, then $\phi$ does imply $\psi$ in $\cF$ (trivially), thus justifying  the terminology.

% $\{\phi_i \to \bigvee_{j\in J_i} \phi_{i j} : i\in I\}$ (where, as above, it is assumed that $\phi_{i j}\leq\phi_i$ for all $i$) such that $\phi_i\leq_\cF\phi_k$  implies that (holds only if), for every $j\in J_i$, the formula $\phi_{ij}$ is equivalent to $\phi_{kl}$ for some $l\in J_k$.
%
% An F-class $\cF$ is \emph{special} if  it has a set of F-axioms $\{\phi_i \to \bigvee_{j\in J_i} \phi_{i j} : i\in I\}$ (where, as above, it is assumed that $\phi_{i j}\leq\phi_i$ for all $i$) such that $\phi_i\leq_\cF\phi_k$  implies that (holds only if), for every $j\in J_i$, the formula $\phi_{ij}$ is equivalent to $\phi_{kl}$ for some $l\in J_k$.
% 
% ???More has to be required??? ALL possible $\phi$!!!???See proof below?
\end{definition}

\begin{rem}
 $_R\flat$ is a special F-class.  This follows from \cite{Zim},  see \cite[Thm.\,2.3.9]{P2} or the explicit axioms in \cite{PRZ2} or \cite[Fact 1.3]{Tf}, which are in standard form---each disjunct of the conclusio implies the premise.
 
  Is the class of torsion-free modules a special F-class?
 It is an F-class, cf.\ \cite{Tf}. However, one has F-axioms only for annihilation formulas $\phi$, so not \emph{every} pp formula may end up being equivalent to such a disjunction. I have no counterexample at hand, but certainly the proof below breaks down  when not \emph{every} pp formula occurs as a premise.

\end{rem}

\begin{lem}\label{pureresolvF}
 Purely resolving classes are special F-classes.
\end{lem}
\begin{proof} As all formulas in $\ppf_\cB\phi$ are below $\phi$, the axioms given in the previous sections are in standard form.

To verify that they are special in the sense of the above definition, consider two axioms $\phi \to \bigvee \ppf_\cB\phi$ and $\psi \to \bigvee \ppf_\cB\psi$ with $\phi\leq_{\tC}\psi$. All we need is $\ppf_\cB\phi\subseteq\ppf_\cB\psi$. 

So let  
$\theta \in \ppf_\cB$ be below $\phi$. To show that it is also below $\psi$, let 
 $(M, \br m)$ realize $\theta$. Pick a free realization  $(B, \br b)$ of $\theta$ with $B\in \cB$, send it to $(M, \br m)$ and observe that it remains to see that $\psi$ is already satisfied in $(B, \br b)$. But this follows from   $B\in\cB\subseteq\tC$ and  the assumption $\phi\leq_{\tC}\psi$.
 \end{proof}

This will be used in \S \ref{basicS}.

\section{Separation}

Reinhold Baer introduced separability in abelian groups---of finite subsets by direct summands that are completely decomposable, see \cite{F}.

%Compare with \cite[\S 7]{MLII}.

\subsection{Separabilities}

We distinguish several types of separation---according to three different coordinates,  cardinality of sets to be separated, the form of separation, and the kind of submodules to perform the separation. In Baer's original separation, the second coordinate, the form of separation, was by direct summand, which is what one usually sees in the literature in the various types, see \cite{EM}. This can be relaxed to Azumaya's locally split,\footnote{also to his finitely split submodules,} or to pure  submodules, called \emph{local separation} and \emph{pure separation}, respectively. On top of that, we will have an entire family of relativized pure separabilities.

Generalizing e.g.\ \cite[\S 4]{GIRT}, we make the following, essentially three, definitions.

\begin{definition}\label{sepdef}
 Let $\kappa$ be an infinite regular cardinal and $\cS$ a class of modules. %Let $s$ stand for either nothing, or `locally' or `pure.'

\begin{enumerate}
 \item  A module $M$ is $\kappa$-\emph{separable by (modules from) $\cS$}, or \emph{$(\kappa, \cS)$-separable} for short, if every subset of $M$ of cardinality $<\kappa$ is contained in a submodule from  $\cS$ that is a direct summand of $M$. 
 \item We say \emph{locally} (resp., \emph{$\cL$-pure}) $\kappa$-\emph{separable by (modules from) $\cS$}, or \emph{locally} (resp., \emph{$\cL$-pure}) \emph{$(\kappa, \cS)$-separable}, if `direct summand' is replaced by `locally split submodule' (resp., by  `$\cL$-pure submodule').
  \item If  $\kappa=\aleph_0$, we simply omit it or say \emph{finite separability} (for whichever version of separability at hand).
 \item  If $\kappa=\aleph_1$, we  speak of  \emph{countable separability}. 
 \item  We omit the reference  to the class $\cS$ when it is $\Rmod$. 
\end{enumerate}
   \end{definition}

By these conventions,  \emph{local separability by $\cS$}, or  \emph{local $\cS$-separability}, is local $(\aleph_0, \cS)$-separability, while  \emph{$\cL$-pure  separability by $\cS$}, or \emph{$\cL$-pure $\cS$-separability}, is $\cL$-pure $(\aleph_0, \cS)$-separability. And, for instance,   a \emph{separable module} is one all of whose finite subsets  are contained in a finitely presented direct summand---beware, this is often handled differently in the literature, cf.\  \cite {EM} and \cite{F}.

\begin{rem}\text{}
\begin{enumerate} [\upshape (a)]
 \item Bear's original separability is (finite) separability---by completely decomposable submodules, cf.\,\cite{F}.  
 \item Clearly, any of the $\kappa$-separabilities entail the corresponding $\lambda$-separability for every $\lambda\leq\kappa$.
 \item Clearly, direct summand $\Longrightarrow$ locally split submodule $\Longrightarrow$ pure submodule $\Longrightarrow$ $\cL$-pure submodule, hence
  %we have the corresponding implications
  
% \noindent
 separable $\Longrightarrow$ locally  separable  $\Longrightarrow$ pure  separable $\Longrightarrow$ $\cL$-pure  separable.
\end{enumerate}

  % And then, as \cite{EM} does, we reserve countably for $<\aleph_1$? 
\end{rem}

%These implications are proper.??? Example???YES , down in the last "basic" examples one should have them all??

 Our interest in separation  lies in the following facts, which we single out for reference.

\begin{lem}\label{basicfact}\text{}
\begin{enumerate}  %[\rm (1)]
 \item If $M$ is  $\cL$-pure separable by $\cL$-atomic modules, $M$ is $\cL$-atomic (hence $\langle\cL\rangle$-atomic).
 
 \item If $M$ is locally  separable by strict $\cL$-atomic modules, $M$ is strict $\cL$-atomic.
 \item  If $M$ is locally  separable by countably generated strict $\cL$-atomic modules, $M$ is strict $\langle\cL\rangle$-atomic.
\end{enumerate}
\end{lem}
\begin{proof}
 (1) By definition $N\subseteq M$ is $\cL$-pure iff the pp types of all tuples in $N$ are $\cL$-equivalent, whether taken in $N$ or in $M$, which is, by the Main Theorem of \cite{habil} or \cite{MLII} all we need.
 
 (2) First separate a  given tuple in $M$ by a locally split submodule $N$ that is strict $\cL$-atomic. Inside $N$, that tuple $\cL$-freely realizes a certain pp formula $\phi$. The same formula generates the tuple's pp type in $M$, for locally split submodules are pure. To see that this tuple freely realizes $\phi$ also in $M$,  simply combine the tuple's local retract from $M$ to $N$ by any of the `free' maps from $N$ to a realization of $\phi$ in any other module from $\cL$.
 
% Which????
% 
% (2) Let $\br m$ be a tuple in $M$. If $N\subseteq M$ is a locally split submodule containing $\br m$ that is  strict $\cL$-atomic, then $(N, \br m)$ $\cL$-freely realizes some pp formula $\phi$ and there is a local retract $g$ sending $(M, \br m)$ to $(N, \br m)$. Then $\pp_M(\br m) = \pp_N(\br m)$ is also generated by $\phi$. If now $L\in\cL$ and some $(L, \br a)$ satisfies $\phi$, there is a map $f: (N, \br m) \to (L, \br a)$. This map combined with $g$ yields the desired map $(M, \br m) \to  (L, \br a)$.
% 
 (3) Apply the (stronger half of the) `Strict' Lemma and (2) above.
\end{proof}

The question of a converse arises at once. This will be addressed in \S \ref{SepPGen}, especially in Theorem \ref{Thm.1}.    %of  how much of a converse to be expected???xxxyyy
We conclude this section with some general observations.

\begin{lem}\label{locsep}
Suppose $\cB\subseteq \cL$ and $M$ a strict $\cL$-atomic module. 

\begin{enumerate}%[\rm (1)]
\item 
The following are equivalent for any infinite cardinal 
$\kappa$.

\begin{enumerate}[ (i)]

 \item $M$ is $\cL$-pure $\kappa$-separable by $\cB$.
 \item $M$ is pure $\kappa$-separable by $\cB$.
 \item$M$ is locally $\kappa$-separable by $\cB$.
\end{enumerate}
\item If all members of $\cB$ are finitely generated, $M$ is $\kappa$-separable by $\cB$ in that case.
\end{enumerate}
\end{lem}
\begin{proof}
 As $\cB\subseteq\cL$,  for submodules that are in $\cB$, $\cL$-purity is the same as purity, Rem.\ref{LpureInL}. So the first two conditions are equivalent (and this does not require atomicity of $M$). To see they imply the third, and for (2), refer to Lemma \ref{locsplitsub}. 
\end{proof}

\subsection{Separation and definable subcategories}
Next we show that in the situation of the previous lemma, $M$ belongs to the definable subcategory generated by $\cL$.

\begin{lem}\label{sepInL} \text{}%Let $\cL$ be a definable subcategory.
 \begin{enumerate}%[\rm (1)]
  \item If $M$ is $\cL$-pure separable by $\cL$ (hence pure separable by $\cL$), then $M\in\langle\cL\rangle$.
 \item  Suppose $\cL$ is axiomatizable by implications of the form $\Phi\to\Psi$, where $\Phi$ is a possibly infinite conjunction of possibly infinite disjunctions of pp formulas and $\Psi$ a possibly infinite disjunction of possibly infinite conjunctions of existential formulas,
 %$\doublewedge\doublevee \theta_i(\br x) \to \sigma(\br x)$, where the $\theta_i$ are pp formulas and $\sigma$ is any existential formula, 
 all in the same finitely many variables. 
 
  If $M$ is $\cL$-pure separable by $\cL$ (hence pure separable by $\cL$), then $M$ satisfies all those implications  that are true in $\cL$.
 \end{enumerate}
\end{lem}
\begin{proof} (1) is the special case of (2) where both $\Phi$ and $\Psi$ are single pp formulas. To prove (2) consider one such implication $\Phi\to\Psi$ and assume the tuple $\br m$ satisfies the antecedent  $\Phi$. We may $\cL$-pure separate it by some $L\in\cL$, which is actually pure in $M$, Remark \ref{LpureInL}.

Now $\br m$ satisfies all conjuncts of $\Phi$ and hence one disjunct for each of these in $M$. Being pp, these latter are, by purity, satisfied by $\br m$ in $L$ as well. It is easy to see that $\br m$ satisfies all of $\Phi$ in $L$. But being in $\cL$, $L$ satisfies the implication, hence  $\br m$ satisfies $\Psi$ in $L$.  Being existential, $\Psi$ it is also true of that very tuple in the original $M$, as desired.
\end{proof}

%\begin{rem}\cite[???]{Tf}\label{flatax}
% As an example, a module is flat provided it is $_R\flat$-pure separable by flat modules. Indicate the shape of the axioms???
%\end{rem}

%I will return to definable subcategories in the last section of this paper, where the question will be discussed whether strict $\cL$-atomicity depends only on the definable subcategory $\cL$ generates.

\section{Countable self-separation}\label{selfsepsect} %in Mittag-Leffler modules}

\subsection{A general fact} %or `The fundamental separation result/fact???'

 As shown in \cite[Prop.\,3.5]{Tf} or \cite[Cor.\,6.5]{MLII}, a module is $\cK$-Mittag-Leffler if and only if each of its countable subsets is contained in a countably generated $\cL$-pure submodule that is $\cK$-Mittag-Leffler (where $\cK$ and $\cL$ are assumed to be mutually dual). Because of the stronger half of the `Strict' Lemma \ref{StrictLemma}, we can say that $\cK$-Mittag-Leffler modules are countably $\cL$-pure separable by strict $\langle\cK\rangle$-Mittag-Leffler modules. (The easy direction also follows from Lemma \ref{basicfact}.)
 This fundamental separation result  characterizing  $\cK$-Mittag-Leffler modules we now formulate as follows. 
 
\begin{fact}\label{ctblesep} %\cite[Prop.\,3.5]{Tf}, . 
%Every countable subset of an $\cL$-atomic module is contained in a countably generated $\cL$-pure submodule that is strict $\cL$-atomic (even strict $\langle\cL\rangle$-atomic).
%
A module is $\cL$-atomic if and only if it is countably $\cL$-pure separable by countably generated strict $\cL$-atomic modules (even strict $\langle\cL\rangle$-atomic modules).

%Because of the parenthetical clause (about the `definable closure'), we can say:
%Every countable subset of a $\cK$-Mittag-Leffler module is contained in a countably generated $\cL$-pure submodule that is $\cK$-Mittag-Leffler (hence strict $\langle\cK\rangle$-Mittag-Leffler)???. %by Lemma \ref{StrictLemma}). 
\end{fact}

\subsection{Countable self-separation}\label{ctbleLsep} 
A principal deficiency in the above general fact is that one has no reason to expect the strict $\cL$-atomic separating submodules to be themselves in $\cL$, even if the original module was. The L\"owenheim-Skolem Theorem of first-order model theory springs to mind as a remedy, and we are going to pursue this in \S\ref{ctblerings}. In order to avoid repeating a mouthful let us make a definition.

\begin{definition}
 We call a class $\cL$ \emph{(countably) self-separating} if
 every $\cL$-atomic module in $\cL$ is countably $\cL$-pure separable by countably generated (strict) $\cL$-atomic modules \emph{in} $\cL$.
\end{definition}

% ???It is the purpose of \S\ref{ctbleLsep}(this may have to go???), ?\S\ref{ctblerings}, and \S\ref{SepPGen}  to single out cases where this is possible. Is it in fact in $\PGen$???? Unions of $B$'s???

\begin{lem} Suppose $\cL$ is self-separating.
\begin{enumerate}
 \item Every strict $\cL$-atomic module in $\cL$ is countably locally separable by countably generated (strict) $\cL$-atomic modules in $\cL$, hence by strict $\langle\cL\rangle$-atomic submodules in $\cL$.
\item Every strict $\cL$-atomic  module in $\cL$ is  strict $\langle\cL\rangle$-atomic.
\end{enumerate}
\end{lem}
\begin{proof}
For (1) use the `Strict' Lemma \ref{StrictLemma} and Lemma \ref{locsep}(1). For (2) use (1) and Lemma \ref{basicfact}(3).
\end{proof}

\subsection{Elementary classes}\label{ctblerings} Recall that a class axiomatized by finitary first-order axioms is called \emph{elementary}.

\begin{lem}
 Suppose $\cL$ is an elementary class and $\kappa=max(|R|, \aleph_0)$. Let $\cB$ be the class of strict $\cL$-atomic modules in $\cL$ of power $\leq\kappa$.

\begin{enumerate}
 \item Every $\cL$-atomic module in $\cL$ is pure $\kappa^+$-separable by $\cL$-atomic modules in $\cL$ of power $\kappa$.
 \item  Every strict $\cL$-atomic module in $\cL$ is locally $\kappa^+$-separable by $\cB$.
\end{enumerate}
\end{lem}
\begin{proof}
 By the L\"owenheim-Skolem Theorem, every subset of a module $M$ of power at most $\kappa$ is contained in an elementary substructure $N$ of $M$ of power   $\kappa$. Then $N$ is in $\cL$ and a pure submodule of $M$, hence  $\cL$-atomic. This proves (1). For (2) apply Lemma \ref{Lpureofstrict} to see that $N$ is strict $\cL$-atomic and thus in $\cB$.
 In that case, Lemma \ref{locsplitsub} yields that $N$  is locally split in $M$.
\end{proof}

\begin{prop}  Suppose $\cL$ is an elementary class and $R$ is countable.

 A module  in $\cL$ is strict $\cL$-atomic  if and only if it is locally countably separable by countable strict $\cL$-atomic modules in $\cL$.
 
 In particular, elementary classes over countable rings are self-separating.
 \end{prop}
\begin{proof}
 One direction follows from the lemma (for $\kappa=\aleph_0$) and the `Strict'Lemma \ref{StrictLemma}.  The other is immediate from (2) of Lemma \ref{basicfact}.
\end{proof}

This can be slightly strengthened  if the axiomatization of $\cL$ is of a specific kind.

\begin{cor} Suppose $R$ is countable and $\cL$ is axiomatizable by implications as in Lemma \ref{sepInL} (2), but finitary (so that L\"owenheim-Skolem applies).

 A module is strict $\cL$-atomic and  in $\cL$ if and only if it is locally countably separable by countable strict $\cL$-atomic modules in $\cL$.
\end{cor}
\begin{proof}
 All that's missing in the proposition is that if $M$ is so separable, it has to be in $\cL$. But this follows from Lemma \ref{sepInL} (2).
\end{proof}

For definable subcategories this can  be reformulated as follows. %(Can't this also be done for more general dualizable finitary(!) implications as in [PRZ]???)

\begin{cor}
Suppose $R$ is countable, $\cL$ is a definable subcategory (i.e., $\cL=\langle\cL\rangle$) and  $\cK$ is its elementary dual.

 A module $M$ is strict $\cL$-atomic (= strict $\cK$-Mittag-Leffler) and a member of $\cL$ if and only if it is countably locally separable by countably generated strict  $\cL$-atomic modules that are members of $\cL$. 
  \qed
 \end{cor}

%\begin{rem}
%One might try to derive strict $\langle \cL\rangle$-atomicity from strict $\cL$-atomicity over a countable ring using this countable local separation, but even if $M\in\cL$, one gets only $N\in \langle \cL\rangle$, so strict $\cL$-atomicity  of $M$ won't necessarily allow one to map $M$ to $N$ as needed to make it locally split.
%\end{rem}
%

%Do we even need this now???? Work with simple $\cL$-free realizations instead????. By Lemma \ref{}, $N$ freely realizes for $\cL$ some  $\omega$-chain of pp formulas that $N$ (considered as a sequence of elements in $M$) freely realizes in $M$. Clearly, $N$ realizes the same sequence in itself (even freely), and so we can map $M$ to $N$ (for $N \in\cL$) while fixing the set $N$ pointwise. We thus obtain local  retracts showing that $N$ is a locally split submodule  of $M$. 
%%\end{proof}

\begin{qu}
 Is this true over uncountable rings?
\end{qu}

%\begin{rem}
% Of course, one can apply the same argument to show that every subset of the power of an uncountable ring, $\kappa$, say, is contained in a locally split submodule of power at most $\kappa$. However, then one looses the guarantee that this submodule is \emph{strict} $\cK$-Mittag-Leffler.  NOOOOOO!!!!??? See above \ref{Lpureofstrict}.
% 
% So any module in $\langle\cL\rangle$  over an infinite ring $R$ is strict $\cK$-Mittag-Leffler if and only if it is $|R|^+$-locally separable by  strict $\cK$-Mittag-Leffler modules of power at most $|R|$. 
%\end{rem}

\subsection{Special F-classes}\label{specialF} For $\cL$, the class $_R\langle\flat\rangle$ of \emph{quasi-flat} modules, the previous proposition implies that, over a countable ring, a quasi-flat module is strict $_R\langle\flat\rangle$-atomic if and only if it is countably separable by countably generated strict $_R\langle\flat\rangle$-atomic modules. However, we can drop the countability assumption on $R$ in this case, as follows from \cite[Prop.3.7]{Tf} and get a better separation result at the same time---namely by flats rather than quasi-flats. Analyzing that proof, we have arrived at the following abstract version that tells the gist of it. We will work in special F-class as introduced in \S\ref{specialFdef} with terminology and notation from Definition \ref{pgen}.
 This is the abstract version of \cite[Prop.\,3.7]{Tf}:

\begin{prop}
Every special F-class is (countably) self-separating.
% 
% Then every $\cF$-atomic module in $\cF$ is countably $\cF$-pure separable by countably generated (strict) $\cF$-atomic modules in $\cF$.
\end{prop}
%\begin{proof} Referring to
% the proof of \cite[Prop.\,3.7]{Tf}, we start also here by modifying the $\cF$-generator $\phi_i$ of the type $\pp_M(a)$ from the proof of \cite[Prop.\,3.5]{Tf}  (= Fact \ref{ctblesep}) by a disjunct $\phi_{ij_a}\in\pp_M(a)$ of the conclusio, which is also an $\cF$-generator of $\pp_M(a)$. Taking witnesses for $\phi_{ij_a}$ in each of the countable steps and then closing off by the countably generated submodule $N$ of $M$, we see, as in \cite[Prop.\,3.7]{Tf}, that $N$ is $\cF$-atomic and $\cF$-pure in $M$. It remains to verify that $N\in\cF$, for which we will verify the axioms. 
% 
% 
% Let $\phi_k \to \bigvee_{j\in J_k} \phi_{k j}$ be one of them. Suppose $a\in N$ satisfies  $\phi_k$. We have to show $a$ also satisfies one of the $ \phi_{k j}$ in $N$. 
% 
% As $\phi_k\in\pp_M(a)$, we have  $\phi_i\leq_\cF \phi_k$. Now, since $\cF$ is special, this can happen only because every disjunct $\phi_{ij}$ of $\phi_i$ occurs as a disjunct of $\phi_k$. In particular, $\phi_{ij_a}$ is among the disjuncts of $\phi_k$, i.e., 
% $\phi_{ij_a}$ is (equivalent to) some 
% $\phi_{k j}$. Hence $a$ satisfies $\phi_{k j}$, as desired.
% \end{proof}
 \begin{proof} Referring to
 the proof of \cite[Prop.\,3.7]{Tf}, we start also here by modifying the $\cF$-generator $\phi$ of the type $\pp_M(a)$ from the proof of \cite[Prop.\,3.5]{Tf}  (= Fact \ref{ctblesep}) by a disjunct $\phi_{i_a}\in\pp_M(a)$ of the conclusio, which too is an $\cF$-generator of $\pp_M(a)$. Taking witnesses for $\phi_{i_a}$ in each of the countable steps and then closing off by the countably generated submodule $N$ of $M$, we see, as in \cite[Prop.\,3.7]{Tf}, that $N$ is $\cF$-atomic and $\cF$-pure in $M$. It remains to verify that $N\in\cF$, for which we will verify the axioms. 
 
Let  $\psi \to \bigvee \psi_j$ 
  be one of them. Suppose $a\in N$ satisfies  $\psi$. We have to show, $a$ also satisfies one of the $ \psi_{j}$ in $N$. 
 
 As $\psi\in\pp_M(a)$, we have  $\phi\leq_\cF \psi$ (even $\phi_{i_a}\leq_\cF \psi$). Since $\cF$ is special, this can happen only if every disjunct $\phi_{i}$ of $\phi$ implies some disjunct of $\psi$. In particular, $\phi_{i_a}$ does, which shows that $a$ satisfies  some disjunct of $\psi$, as desired.
 \end{proof}

\begin{cor}\label{pureresolvself} Every purely resolving class is (countably) self-separating.  
\end{cor}
\begin{proof}
Use Lemma \ref{pureresolvF}.
\end{proof}
% Then every $\tC$-atomic module in $\tC$ is countably $\tC$-pure separable by countably generated (strict) $\tC$-atomic modules in $\tC$.

\begin{cor}\cite[Prop.\,3.7]{Tf}
Every flat $_R\flat$-atomic module is countably $_R\flat$-pure separable by countably generated flat (strict) $_R\flat$-atomic modules.%\qed
\end{cor}
\begin{proof}
By Corollary \ref{flatISstrict}, $_R\flat$ is purely resolving.
\end{proof}
%What is the effect of those countably generated strict $\flat$-atomic modules actually being strict (quasi-flat)-atomic?????
%
%Compare also with Remark \ref{flatax} above???

%Question: is the same true for torsion-free instead of flat?????

%\subsection{$R_R$-Mittag-Leffler modules}???here??? No, no, that's the flat case down below. In that case we can get $\cL$-pure separation for $\cL-\omega$-limits by $\Rmod$. EXPLORE!!!!???
%
\subsection{Relative Mittag-Leffler modules in purely resolving classes} Recall from Corollary \ref{pureresolvISstrict}: strict $\cL$-atomic modules are strict atomic---i.e., strict Mittag-Leffler---provided they are members of $\cL$ and $\cL$ is purely resolving. We are going to show the same for the non-stict version. 

\begin{thm}\label{LMLisML}
 Suppose $\cL$ is purely resolving.
 
 $\cL$-atomic modules in $\cL$ are Mittag-Leffler.
\end{thm}
\begin{proof}
Let $M$ be   $\cL$-atomic. By Corollary \ref{pureresolvself}, $M$ is countably $\cL$-pure separable by strict $\cL$-atomic modules in $\cL$. As these are strict atomic and $\cL$-purity is just purity by Remark \ref{LpureInL}, the module $M$ is pure separable by atomic modules, hence atomic itself, which means Mittag-Leffler. 
\end{proof}

\begin{cor} \cite[Thm.\,3.10]{Tf}
 Flat $\flat$-atomic modules are Mittag-Leffler.\qed
\end{cor}

\section{Finite separation in
purely generated classes}\label{SepPGen}

\subsection{The first separation theorem}

\begin{definition}\textbf{}

\begin{enumerate}
 \item A \emph{separation pair} is an ordered pair $(\cL, \cB)$ of classes of modules, $\cB\subseteq\cL$, where $\cL$ is  closed under direct sum and $\cL$-pure epimorphic images and $\cB$ is a \emph{small}\,\footnote{I.e., containing only a set of distinct isomorphism types, in other words, $\cB/$$\cong$ is a set.} class of strict $\cL$-atomic modules closed under finite direct sum.

 %$\cL$ in a separation pair $(\cL, \cB)$ we also call  a \emph{$\cB$-separation class}. ????don't think so
 
 \item  A \emph{$\cB$-free module} is simply a direct sum of modules from $\cB$. The class of all of them is denoted $\oplus\cB$.
\item A \emph{large $\cB$-free module} is a direct sum of direct powers of infinitely many copies of each isomorphism type of modules from $\cB$, i.e., a module of the form $\oplus_{B\in\cB^*} B^{(\kappa_\cB)}$ with all $\kappa_\cB$ infinite, where $\cB^*$ is any transversal in $\cB/$$\cong$.
\item We denote by $P^*$ the large $\cB$-free module $\oplus_{B\in\cB^*} B^{(\omega)}$. 
%The one large $\cB$-free module that has all $\kappa_\cB$ countably infinite (or $\omega$) is denoted by  $P^*$.
\end{enumerate}
\end{definition} 

 %In other words, a large  $\cB$-free module is a module of the form $\oplus_{B\in\cB^*} B^{(\kappa_\cB)}$ with all $\kappa_\cB$ infinite, where $\cB^*$ is any transversal in $\cB/$$\cong$, and $P^*=\oplus_{B\in\cB^*} B^{(\omega)}$. 
 
\begin{notat}\label{setting}  Throughout this subsection,  $(\cL, \cB)$ is a separation pair. Following \cite[\S 2]{HR},  we let $\cC=\Add\cB$  and  $\tC$ be the class of all pure epimorphic images of modules from $\cC$ (equivalently, from $\oplus\cB$),  i.e., the class purely generated by $\cB$, which  is therefore also denoted $\PGen\cB$.
\end{notat}

\begin{rem}\text{}
\begin{enumerate}[\upshape (a)]
 \item As $\cB$ is closed under finite direct sum, $\cB$-free modules are (finitely) $\cB$-separable.
 \item Relaxing pure generation to $\cL$-pure generation would not gain one anything, for $\cC\subseteq\cL$, and $\cL$-pure epimorphisms emanating from modules in $\cL$ are pure epimorphisms.
\end{enumerate}

\end{rem}

\begin{lem}\label{Add}  \cite[Prop.2.5]{Tf}. %Suppose $(\cL, \cB)$ is a separation pair.
 Every countably generated (strict) $\cL$-atomic module  in $\PGen\cB$ is contained in $\Add \cB$, i.e., a direct summand of a $\cB$-free module (which can be taken to be large) and thus a direct summand of a module (finitely) separable by  $\cB$.
 \end{lem}
\begin{proof}
 Being in $\PGen\cB$, the module $N$ is a pure epimorphic image of a $\cB$-free module $P$. But pure epimorphisms emanating from $\cL$---this is where $\cB\subseteq\cL$ is needed---onto countably generated  $\cL$-atomic modules split by \cite[Prop.2.5]{Tf}. %do we have this here somewhere???
 \end{proof}

\begin{lem}[Eilenberg's Trick]
For every  $\cL$-atomic module $N\in \PGen\cB$ there is a large $\cB$-free module $P$ such that $N\oplus P \cong P$. 

If $N$ is, in addition, countably generated, then $N\oplus P^*\cong P^*$.
 \end{lem}
 \begin{proof}  By the previous lemma, we have  $P=N\oplus M$ for some $M$. Now comes Eilenberg's trick: as $P^{(\omega)}\cong P$ we may rewrite  $P$ as $(N\oplus M)^{(\omega)}\cong N\oplus (M\oplus N)^{(\omega)}\cong N\oplus P$. In case $N$ is countably generated, $P=P^*$ works.
\end{proof}

%Here comes the key fact.???

\begin{lem} [The Separation Lemma]\textbf{} %Suppose $(\cL, \cB)$ is a separation pair.
If $M$ is an $\cL$-atomic module in $\PGen\cB$, then $M\oplus P^*$ is  $\cL$-pure separable by  $\cB$.
\end{lem}
\begin{proof}
  Consider a finite subset (a tuple) of $M\oplus P^{*}$ and denote its projections onto $M$ and $P^{*}$ by $\br m$ and $\br p$, respectively. As $P^{*}$ is separable by modules from $\cB$, the tuple $\br p$ is contained in a  direct summand $B\in\cB$ of $P^{*}$. By the basic properties of $P^{*}$, we can split it into $P^{*}_{0} \oplus B$, where $P^{*}_{0}\cong P^{*}$, and then write $M\oplus P^{*}=M\oplus P^{*}_{0} \oplus B$. It remains to $\cL$-pure-separate $\br m$ in $M\oplus P^{*}_{0}\cong M\oplus P^{*}$ by a  module from $\cB$.

By hypothesis, we can $\cL$-pure separate  $\br m$ in $M$ by a countably generated $\cL$-atomic module $N\in \PGen\cB$, an $\cL$-pure submodule of $M$.

Then $N\oplus P^{*}$ is an $\cL$-pure submodule of $M\oplus P^{*}$ containing $\br m$, and it therefore suffices to $\cL$-pure-separate $\br m$ by a  module from $\cB$ in $N\oplus P^{*}$. 
Now Eilenberg's trick shows us that $N\oplus P^{*}$ is isomorphic to $P^*$ and thus (even) separable by  $\cB$. 
\end{proof}

\begin{thm}\label{Thm.1} Suppose $(\cL, \cB)$ is a separation pair.

\begin{enumerate}%[\rm (1)]
 \item  Every $\cL$-atomic module in $\PGen\cB$ is a  direct summand of a module that is (finitely) $\cL$-pure separable by $\cB$ (and conversely).

 \item Every  strict $\cL$-atomic module in $\PGen\cB$ is a   direct summand of  a module that is locally separable by $\cB$ (and conversely).

\end{enumerate}
\end{thm}
\begin{proof}
The direction from left to right in (1) follows from the Separation Lemma. Namely, if $M$ is $\cL$-atomic, $M\oplus P^*$ is $\cL$-pure separable by $\cB$. For (2), notice that then  $M\oplus P^*$ is even strict $\cL$-atomic, hence even locally $\cB$-separable by Lemma \ref{locsep}.

The converses follow from Lemma \ref{basicfact} (together with  Lemma \ref{Lpureofstrict}).
\end{proof}

See Corollary \ref{classicalThm.1} for the case of the classical separation pair $(\RMod, \Rmod)$.

%???What about a cor here on the countably generated case???

\begin{cor}\label{f.g.B:1} Suppose $(\cL, \cB)$ is a separation pair with all modules in $\cB$  finitely generated.

 Then every strict $\cL$-atomic module in $\PGen\cB$ is a direct summand of  a (finitely) $\cB$-separable module (and conversely).

\end{cor}
\begin{proof}
Employ Lemma \ref{locsep}(2).
\end{proof}

\begin{rem} \text{}

\begin{enumerate}[\upshape (a)]
 \item One may be tempted to think that the following is true even without $\cB\subseteq\cL$:\\
  $\cL$-atomic modules in $\PGen\cB$ are  direct summands of modules that are (finitely) $\cL$-pure separable by $\cB$.
 However, the entire series of lemmas leading to the theorem is based on the first one, Lemma \ref{Add}, which does need it. %???Check??? And then incorporate in the statements.???
 \item   The converses of the theorem are true only within $\PGen\cB$, for the separabilities in question only imply inclusion in $\cL$, not necessarily in  $\PGen\cB$.
 \item Remember also, that $\PGen\cB\subseteq\cL$, so the theorems are only about $\cL$-atomic modules belonging to $\cL$ (even $\PGen\cB$, which may be less). %Provide example????(flats???))
\end{enumerate}
\end{rem}

\subsection{Finite separation and direct sum decomposition} 
Baer made exactly the following argument---in his special case. 

\begin{lem}\label{countablesplit}{\rm [Baer's Lemma]}
 Let $\cG$ be a class of modules closed under direct summand  and such that  every module in  $\cG$  is (finitely) $\cB$-separable.
\begin{enumerate}
\item  Every module in  $\cG$  is countably locally separable by (countable) direct sums of modules from $\cB$.

\item Countably generated modules from   $\cG$ are direct sums of modules from $\cB$.
\end{enumerate}

\end{lem}
\begin{proof} Let $C$ be a countable subset in $G\in\cG$. 
(1) List the set $C$ as $c_{i}$ ($i < \omega$).  Set $G_{0}=G$ and $b_{0}=c_{0}$.  As $G_{0}$ is $\cB$-separable, $b_{0}$ is contained in a direct summand $B_{0}\in \cB$ of $G_{0}$. Write $G_{0}= B_{0}\oplus G_{1}$ and $c_{1}=a_{1} + b_{1}$ accordingly. Since $G_{1}$ is again in $\cG$, hence $\cB$-separable, $b_{1}$ is contained in a direct summand $B_{1}\in \cB$ of $G_{1}$.  Successively we decompose $G$ into $\oplus_{i<n}B_i\oplus G_{n}$ where each $B_{i}$ is in $\cB$ and so that the sum of the first $n$ summands $B_{i}$ contains the first $n$ elements of $C$. Eventually, $C$ is contained in the direct sum $B_\omega:=\oplus_{n < \omega}B_n$. And while $B_\omega$  may not be a direct summand of $G$, it certainly is a locally split submodule. %Besides, it is in $\cL$.

(2) If $C$ happens to generate all of $G$, then inevitably, $G=B_\omega$.
\end{proof}

%\begin{cor}
% If every module in $\Add(\cB)$ is (finitely) $\cB$-separable, then every countably generated module in $\Add(\cB)$ is a direct sum of modules from $\cB$.
%\end{cor}
%(for instance, $\cG=\Add(\cB)$)

Applying this to $\cG=\Add(\cB)$, the situation considerably simplifies  when all the generating modules are countably generated, cf.\ \cite[Lemma 7]{GIRT}.

\begin{cor}
 Suppose $\cB$ is a class of countably generated  modules closed under finite direct sum.

 The following assertions are equivalent. %???for any infinite cardinal $\kappa$.
\begin{enumerate}[(i)]
\item Every module in $\Add(\cB)$ is (finitely) $\cB$-separable.
\item Every module in $\Add(\cB)$ is a direct sum of modules from $\cB$, i.e., $\Add(\cB)=\oplus \cB$.

%\item ???Want this??? Every  $(\kappa, \cC)$-separable module (in particular, every module in $\cC$)  is $\cB$-separable.
\end{enumerate}
\end{cor}
\begin{proof} 
If $\cB$ is closed under finite direct sum, all modules in $\oplus \cB$ are $\cB$-separable. This yields the easy direction. For the nontrivial direction, note that by Kaplansky's Theorem---which applies due to the hypothesis made on $\cB$---\emph{every} module in $\Add(\cB)$ decomposes  into a direct sum of countably generated modules in $\Add(\cB)$.
\end{proof}

%The instance of this where $\cB\subseteq\Rmod$ was  ,

\subsection{Back to countable self-separation}  Applying Baer's Lemma to the class of all strict $\cL$-atomic modules in $\PGen\cB$ allows one to  considerably strengthen the fundamental countable separation of Fact \ref{ctblesep} as follows. (Note, from Fact \ref{ctblesep} alone we would not know if the separating module from $\cL$ was in $\PGen\cB$, let alone that it was $\cB$-separable.)

\begin{cor}\label{selfsep}
 Suppose $(\cL, \cB)$ is a separation pair and every strict $\cL$-atomic module in $\PGen\cB$ is $\cB$-separable. (This is the case, for instance, if (iii) of \ref{Thm.2} holds and $\cB$ consists of finitely generated modules only.)

\begin{enumerate}
 \item Then all strict $\cL$-atomic modules in $\PGen\cB$  are countably locally separable by countable direct sums of modules in $\cB$.
 \item In particular,  all strict $\cL$-atomic modules in $\PGen\cB$  are countably locally $\cL$-separable.% (by modules from $\PGen\cB$, even from $\oplus\cB$, hence by $\cB$-separable modules???). \qed
 \end{enumerate}
\end{cor}

\subsection{The second separation theorem:  when everything is `basic'}\label{basicS} By `basic' we mean that direct summands can be avoided. For instance, we wish to investigate when $\Add\cB=\oplus \cB$ and what the effect of that is on the conclusions of Theorem \ref{Thm.1}. In particular, we are interested to see in which cases those will be `basic' too.
%
%check how much $\PGen\cB$ is needed????
%
%CHECK, in fact, how to SALVAGE this theorem other than with 

\begin{thm}\label{Thm.2}   Suppose  $(\cL, \cB)$ is  a separation pair. 

\begin{enumerate}
 \item The following are equivalent. 
 \begin{enumerate} [(i)]
 \item Every module in $\Add\cB$ is $\cL$-pure separable by $\cB$.

 \item Every module in $\Add\cB$ is  locally separable by $\cB$.

 \item Every (countably generated) strict $\cL$-atomic module in $\tC$  is locally separable by $\cB$.
 \end{enumerate}
\item  The next statement implies the above ones. It is equivalent to them, provided $\cL=\tC$ (= $\PGen\cB$) and $\tC$ is self-separating.

(iv)  Every $\tC$-atomic module in $\tC$  is  $\tC$-pure separable by $\cB$.
  \end{enumerate}
 %Compare with Cor.\ref{selfsep}???
\end{thm}
\begin{proof}
  (iv) $\Rightarrow$ (iii) follows from Lemma \ref{locsep}, while  (ii) is a special case of (iii), for every module in $\Add\cB$ is strict $\tC$-atomic and    $\Add\cB\subseteq\PGen\cB$.
    And  (i) and (ii) are equivalent by Lemmas \ref{basicfact}  and \ref{locsep}. We are left with  (ii) $\Rightarrow$ (iv), the only implication where we seem to have to invoke the extra hypotheses of $\cL$ being   self-separating and of it being equal to $\tC$.
    
  For this, first note that (ii) implies, together with Lemma \ref{Add}, the countably generated part of (iii). Now consider a tuple $\br m$ in a $\tC$-atomic module $M\in\tC$. By hypothesis, $\br m$ is contained in a countably generated $\tC$-pure submodule $N\in\tC$ of  $M$ that is strict $\tC$-atomic, hence, by what has just been said, locally $\cB$-separable. Consequently, $\br m$ is contained in a $\tC$-pure submodule  $B\in\cB$ of $M$, as desired.
%
%  (i) $\Rightarrow$ (iii): By countable separation, 
%Fact \ref{ctblesep}, an $\cL$-atomic module is (finitely, even countably) $\cL$-pure separable by countably generated (strict) $\cL$-atomic modules
%
%CAREFUL???? These may not be in $\PGen\cB$, so need Corollary \ref{selfsep}???
%
%, hence, by \ref{Add}, so separated by modules from $\Add \cB$. But if these are actually  $\cL$-pure  separable by $\cB$, so are all $\cL$-atomic modules.
% 
% (iii) $\Rightarrow$ (iv) (does not need $N\in\cL$???): By Fact \ref{ctblesep}, any given tuple $\br m$ in $M$ is contained in an $\cL$-pure submodule $N$ of $M$ that is strict $\cL$-atomic. By (ii), $N$ is locally $\cB$-separable by some $B\in\cB$, a pure submodule of $N$. As  $B$ is also $\cL$-pure in $M$, this concludes the proof.
% 
%   
%
%  ???Eliminate redundant direction ((iii) $\Rightarrow$ (iv)???), but keep that for book (exercise???).
\end{proof}

\begin{rem}
% The proof shows that the statements (i) through (iii) are equivalent for \emph{any} separation pair $(\cL, \cB)$ with $\cL$ instead of the more restrictive $\tC$ (but still only for those modules \emph{in} $\tC$!).
% And similarly, (iv) implies these also then. It is only for the implication \emph{to} (iv) that we need $\cL$ to be $\tC$ and self-separating. 
% 
 Finite self-separability of $\cL$ suffices in (iv).
\end{rem}

Corollary \ref{pureresolvself} now yields

\begin{cor}
 Statements (i) through (iv) of the theorem are equivalent  for any purely resolving class $\cL=\tC$.
(But note Theorem \ref{LMLisML}.)
 \qed
\end{cor}

When all members of $\cB$ are countably generated, Kaplansky's Theorem allows for a slightly different and stronger version of the previous theorem.

\begin{thm}\label{c.g.B:2} Suppose  $(\cL, \cB)$ is  a separation pair with all modules in $\cB$ countably generated.

\begin{enumerate}
 \item The following are equivalent. 
\begin{enumerate} [(i)]
 \item Every module in $\Add\cB$ is separable by $\cB$.
 \item Every module in $\Add\cB$ is countably locally separable by $\oplus\cB$.
 
 \item $\Add\cB = \oplus\cB$.
 
 \item Every countably generated (strict) $\cL$-atomic module in $\PGen\cB$ is in $\oplus\cB$.

 \item Every strict $\cL$-atomic module in $\PGen\cB$ is countably locally separable by $\oplus\cB$.
\end{enumerate} 
\item The following statement implies the previous ones. It is equivalent to them, provided $\cL=\tC$ (= $\PGen\cB$) and $\tC$ is self-separating.

(vi) Every $\tC$-atomic module in $\tC$ is countably  $\tC$-pure separable by $\oplus\cB$.

 \end{enumerate}
\end{thm}
\begin{proof}
 (i) $\Rightarrow$ (ii) is Baer's Lemma \ref{countablesplit}. 
 
 (ii) $\Rightarrow$ (iii): first note that (ii) implies that \emph{countably generated} modules in $\Add\cB$ are in $\oplus\cB$. To see the same for \emph{all} modules in $\Add\cB$, 
 apply Kaplansky's Theorem, which shows that every module in $\Add\cB$ is a direct sum of countably generated modules from $\Add\cB$ (here the countable generation of all members of $\cB$ is used).
 
  (iii) $\Rightarrow$ (iv) follows from Lemma \ref{Add},   (iv) $\Rightarrow$ (vi) from the basic Fact \ref{ctblesep}, and  (vi) $\Rightarrow$ (v) from Lemma \ref{locsep}, while (v) $\Rightarrow$ (iv) is trivial.
  
  Finally, (iv) $\Rightarrow$ (iii) is by Kaplansky's Theorem again.
   % ???is Cor. \ref{selfsep} a shortcut anywhere here???
\end{proof}

\begin{rem}
% To be incorporated??? we have the following finite separations in the above thm each time, but are they also equivalent???? Note, (v)' and (vi)' are ALWAYS equivalent.
% 
 \noindent
 (iv) above implies (iv)': every countably generated (strict) $\cL$-atomic module in $\PGen\cB$ is separable by $\cB$. 
  
 \noindent
 (v) above implies (v)': every strict $\cL$-atomic module in $\PGen\cB$ is locally separable by $\cB$. 
 
 \noindent
 (vi) above implies (vi)': every $\cL$-atomic module in $\PGen\cB$ is $\cL$-pure separable by $\cB$.
 
 Note, (v)' and (vi)' are  equivalent.
\end{rem}

\begin{cor}\label{f.g.B:2} If all modules in $\cB$ are finitely generated,  one may replace (in Theorem \ref{c.g.B:2}),

in (v): `countably locally separable by $\oplus\cB$' by `(finitely) separable by $\cB$,'

in (vi): `countably  $\tC$-pure separable by $\oplus\cB$' by `(finitely) $\tC$-pure separable by $\cB$'.

% \begin{enumerate} [(i)]
%  \item Every module in $\Add\cB$ is separable by $\cB$.
% \item Every module in $\Add\cB$ is countably locally separable by $\oplus\cB$.
% 
% \item $\Add\cB = \oplus\cB$.
% 
% \item Every countably generated (strict) $\cL$-atomic module in $\PGen\cB$ is in $\oplus\cB$.
%
% \item Every strict $\cL$-atomic module in $\PGen\cB$ is (finitely)  separable by $\cB$.
% 
%\item Every $\cL$-atomic module in $\PGen\cB$ is (finitely) $\cL$-pure separable by $\cB$.
% \end{enumerate}
\end{cor}
\begin{proof} Only (v) and (vi) differ from the corresponding statements in the theorem. In fact, they are (v)' and (vi)' from the remark above, which are equivalent to each other and definitely follow from (v) and (vi) in the theorem. But now that $\cB$ consists entirely of finitely generated modules, (v)' becomes (v) here. Finally, (i) is a special case of (v), so the circle is completed.  %CHECK if (vi) is in the circle????
%Scratch:????  First of all, Lemma \ref{locsep} (2) shows that, under the current hypothesis on $\cB$, (ii) of the theorem reads like (iii) here.
%  (iii) $\Rightarrow$ (i) follows from Lemma \ref{GIRT}, the rest from the theorem.
%   (iii) $\Rightarrow$ (ii) follows from Lemma \ref{countablesplit}.
\end{proof}

\subsection{The classical case} Here we draw the conclusions of the above theorems for the classical case when $\cK=\ModR=\langle \modR\rangle$ and $\cL=\RMod=\langle \Rmod\rangle$. This corresponds to the separation pair $(\RMod, \Rmod)$, or the purely resolving class $\RMod$. (Remember, (strict) Mittag-Leffler is the same as (strict) atomic.)

Recall our convention: `separable' means `separable by $\Rmod$'  here (i.e., separable by finitely presented modules).
%\begin{proof}
%%  (ii) $\Rightarrow$ (i) follows from Lemma \ref{GIRT}, the rest from the theorem.
%\end{proof}

\begin{cor}\label{classicalThm.1}\text{}%\cite[???]{GIRT}\label{basicGIRT} The following are equivalent.

 \begin{enumerate} %[\rm (1)]
\item Every pure-projective $R$-module is a direct summand of a direct sum of finitely presented modules.  {\rm (This  well-known fact is included only to draw the analogy in the next corollaries.)}

 \item \cite[Lemma 3.6(ii)]{HT} Every strict Mittag-Leffler module  is  a direct summand of a separable module.
 
\item Every Mittag-Leffler module  is  a direct summand of a pure separable module.
 \end{enumerate}
\end{cor}
\begin{proof}
Apply Corollary \ref{f.g.B:1} and Theorem \ref{Thm.1}(1), resp., with $\cB=\Rmod$, $\cL=\RMod$. Recall, strict atomic = strict Mittag-Leffler and atomic = Mittag-Leffler.
\end{proof}

The first two statements in the next corollary were proved to be equivalent in \cite[Thm.\,4.2]{GIRT}. %Recall that we use `separable' to mean `separable by finitely presented modules,' i.e., by $\Rmod$.
It says that if `direct summand' can be omitted in any of the three statements of Corollary \ref{classicalThm.1}, it can be omitted in the other two.

\begin{cor}\label{basicGIRT} The following are equivalent.

 \begin{enumerate} [(i)]
 \item Every pure-projective $R$-module is a direct sum of finitely presented modules.
 
 %\item Every countably generated (strict) Mittag-Leffler module  is a direct sum of finitely presented modules.
 
 \item Every strict Mittag-Leffler module  \emph{is} separable.
 
\item Every Mittag-Leffler module  \emph{is} pure separable.
 \end{enumerate}
\end{cor}
\begin{proof}
 Take $\cB=\Rmod$ and $\cL=\RMod$ in Corollary \ref{f.g.B:2}.
\end{proof}

This is the case over $R=\mathbb Z$ (that is, for abelian groups), over Krull-Schmidt rings, over rings that are one-sided noetherian and two-sided serial, as well as over (two-sided) hereditary noetherian rings, cf.\ \cite{GIRT} or \cite{PR} for references.

%\subsection{The case $R_R\in\langle\cK\rangle$} This is equivalent to saying that $_R\sharp\subseteq \langle\cL\rangle$, \cite[\S 5.3]{MLII}. In this case, by Theorem \cite[7.14\,(2)]{MLII}, every countably generated $\cL$-atomic module is a union of an $\cL$-pure  chain of finitely presented submodules. In particular, they are $\cL$-pure separable.
%
%What countable separability is implied???

%\subsection{The case $\langle R_R\rangle=\langle \cK\rangle$} or, equivalently, $\langle \sharp_R\rangle=\langle \cL\rangle$???

\subsection{The flat case}\label{flatcase} %$\cL={_R\flat}$} %$\langle \flat_R\rangle=\langle \cL\rangle$} 

Corollary \ref{f.g.B:1} now reads: Every flat strict $\flat$-atomic module is a direct summand of a $\cB$-separable module (and conversely), where we can take $\cB=\Rproj$, the finitely generated projectives, or $\cB=\Rfree$, the free modules (of finite rank). However,
 we know from Corollary \ref{flatISstrict} already that flat strict $\flat$-atomic are even strict Mittag-Leffler, and similarly, flat $\flat$-atomic (= $\sharp_R$-Mittag-Leffler) modules are Mittag-Leffler, see \cite[Thm.\,3.10]{Tf}.

\begin{cor}\text{}
\begin{enumerate}
 \item Flat strict Mittag-Leffler modules are direct summands of $\Rfree$-separable modules.
 \item Flat  Mittag-Leffler modules are direct summands of modules that are $\flat$-pure $\Rfree$-separable.
\end{enumerate}
\end{cor}

 When we apply Theorem \ref{Thm.2} in the shape of Corollary \ref{f.g.B:2} to this situation, we get two different results, one for $\cB=\Rfree$, one for $\cB=\Rproj$.
Only the last condition in either one is new though, the other two had been achieved already in \cite[Thms.\,5.1 and 5.3]{GIRT}.

\begin{cor}\label{proj=free} The following are equivalent.

 \begin{enumerate} [(i)]
 \item Every projective $R$-module is free.
 
 %\item Every countably generated (strict) Mittag-Leffler module  is a direct sum of finitely presented modules.
 
 \item Every flat strict Mittag-Leffler module is  $\Rfree$-separable.
 
\item Every  flat Mittag-Leffler module  is  $\flat$-pure $\Rfree$-separable.
 \end{enumerate}

\end{cor}

\begin{cor} The following are equivalent.

\begin{enumerate} [(i)]
 \item Every projective $R$-module is a direct sum of finitely generated projective modules.
 
 %\item Every countably generated (strict) Mittag-Leffler module  is a direct sum of finitely presented modules.
 
 \item Every flat  strict Mittag-Leffler module is  $\Rproj$-separable.
 
\item Every  flat Mittag-Leffler module  is  $\flat$-pure $\Rproj$-separable.
 \end{enumerate}
\end{cor}

In \cite[\S5]{GIRT} one may find  lists of rings over which either of these cases take place. For instance, the case of Cor.\ref{proj=free} takes place over the integers. So, a torsionfree (=flat) abelian group is strict Mittag-Leffler iff it is $\Rfree$-separable (such groups are simply called `separable' in \cite[Def.2.4 and after]{EM}). Correspondingly, \cite[Prop.\,7]{AF} implies that a torsionfree abelian group is Mittag-Leffler iff it is $\aleph_1$-free, which is to say that every countable subgroup is free. 

\begin{exam}\label{counterex}
 \cite[Exercise IV.11]{EM} says that ${\mathbb{Z}}^{<\aleph_1}$, the subgroup of ${\mathbb{Z}}^{\aleph_1}$ of all elements with countable support, is $\aleph_1$-free, but not separable (not even torsionless). Thus this group is Mittag-Leffler, but not strict Mittag-Leffler.
\end{exam}

%\subsection{Axiomatizable classes (especially all definable subcategories) over a countable ring $R$} ???Check??? Consolidate with Section \ref{ctblerings}.???

\section{When strict $\cL$-atomic modules \emph{are} strict $\langle\cL\rangle$-atomic}\label{lastsect}

Recall from the referee's Remark \ref{ref} that there are strict $\cL$-atomic modules which are not strict $\langle\cL\rangle$-atomic.
 I conclude with a proof that in the context of the previous section, the opposite is the case.

\begin{prop}
 Suppose $(\cL, \cB)$ is a separation pair such that all modules in $\cB$ are strict $\langle\cL\rangle$-atomic (e.g., when all modules in $\cB$ are countably generated).

Every strict $\cL$-atomic module in $\PGen\cB$  is strict $\langle\cL\rangle$-atomic.
\end{prop}
\begin{proof} Recall from the `Strict' Lemma \ref{StrictLemma} that countably generated strict $\cL$-atomic modules are strict $\langle\cL\rangle$-atomic (which explains the parenthetical clause).

Let now $M\in\PGen\cB$ be strict $\cL$-atomic.  By Theorem \ref{Thm.1}(2), $M$ is a   direct summand of  a module that is locally separable by $\cB$. By hypothesis and Lemma \ref{basicfact} (2)---applied to $\langle\cL\rangle$ rather than $\cL$, the module $M$ is a direct summand of a strict $\langle\cL\rangle$-atomic module, hence itself $\langle\cL\rangle$-atomic.
\end{proof}

\begin{cor}
 Strict $_R\flat$-atomic %(or strict $\sharp_R$-Mittag-Leffler) 
 modules are strict $\langle{_R\flat}\rangle$-atomic, i.e., strict $\langle{\sharp_R}\rangle$-Mittag-Leffler.
\end{cor}

%\begin{rem} This should go with the truncated lemma??? of preservation under sum and pure epi.
%
%The trouble is that it is unclear (to me) how to extend mappability to $M$ to any pure submodule of $M$.  ???Now look at $\cP$-local splitting.???
% \end{rem}

%\section{Examples: summary}
%\subsection{Flat modules}
%As has been pointed out in \cite[\S 2, before Thm.\,2.1]{HR}, a theorem of Lenzing ref??? shows that the class of flat modules is purely generated by the class of finitely generated projectives. Let therefore, throughout this sub???section, $\cB=\Rproj$, the class of finitely generated projective $R$-modules. Then $\PGen\cB = {_R\flat}$. Note, $(\Rproj, {_R\flat})$ forms a separation pair.
%
%Maybe throw ALL of $\flat$ in here.???
%
%and finally the stuff about projective-separable and free-separable.
%
%\subsection{Classes purely generated by a single locally pure-projective module} Can we mine something from [PR]????
\bibliographystyle{amsplain}

\end{document}